\documentclass[12pt]{article}
\usepackage{verbatim}
\usepackage{amssymb}
\usepackage{amsthm}
\usepackage{amsmath}
\usepackage[margin=1in]{geometry}
\usepackage[usenames]{color}
\numberwithin{equation}{section}

\usepackage{latexsym}
\usepackage{amssymb}
\usepackage{euscript}
\usepackage[dvips]{graphics}
\usepackage{epsf}
\let\cal=\mathcal      

\def\mcc{M\raise.5ex\hbox{c}C}
\def\mccarthy{M\raise.5ex\hbox{c}Carthy}

\def\eg{{\it e.g. }}
\def\ie{{\it i.e. }}

\def\h{{\cal H}}

\def\M{{\cal M}}

\def\m{Mult}

\def\l{\lambda}

\def\vare{\varepsilon}

\let\i=\infty

\def\la{\langle}
\def\ra{\rangle}
\def\={\ = \ }    
\def\ot{\otimes}

\def\C{\mathbb C}

\def\D{\mathbb D}

\def\inn{\ \in \ }

\def\be{\setcounter{equation}{\value{theorem}} \begin{equation}}
\def\ee{\end{equation} \addtocounter{theorem}{1}}
\def\beq{\begin{eqnarray*}}
\def\eeq{\end{eqnarray*}}

\def\vs{\vskip 5pt}

\def\bp{{\sc Proof: }}
\def\ep{{}{\hfill $\Box$} \vskip 5pt \par}

\def\bl{\begin{lemma}}
\def\el{\end{lemma}}
\def\bt{\begin{theorem}}
\def\et{\end{theorem}}
\def\bprop{\begin{prop}}
\def\eprop{\end{prop}}
\def\bd{\begin{definition}}
\def\ed{\end{definition}}
\def\br{\begin{remark}}
\def\er{\end{remark}}
\def\bexer{\begin{exercise}}
\def\eexer{\end{exercise}}
\def\bfig{\begin{figure}}
\def\efig{\end{figure}}

\title{Global holomorphic functions in several non-commuting variables}
\author{Jim Agler
\thanks{Partially supported by National Science Foundation Grant
DMS 1068830}
\\ U.C. San Diego\\ La Jolla, CA 92093
\and
John E. M\raise.5ex\hbox{c}Carthy
\thanks{Partially supported by National Science Foundation Grants DMS 0966845 and
DMS 1300280}
\\ Washington University\\ St. Louis, MO 63130}

\def\be{\begin{equation}}
\def\ee{\end{equation}}
\def\norm#1{\| #1 \|}
\def\ip#1#2{\langle #1,#2 \rangle}
\def\m{\mathbb{M}}
\def\md{\mathbb{M}^d}
\def\mm{\mathbb{M}_m}
\def\mn{\mathbb{M}_n}
\def\mnd{\mathbb{M}_n^d}
\def\unin{\mathcal{U}_n}
\def\invn{\mathcal{I}_n}
\def\invm{\mathcal{I}_m}
\def\c{\mathbb{C}}
\def\D{\mathbb{D}}
\def\r{\mathbb{R}}
\def\n{\mathbb{N}}
\def\cn{{\mathbb{C}^n}}
\def\k{\mathcal{K}}
\def\h{\mathcal{H}}
\def\s{\mathcal{S}}
\def\lhk{\mathcal{L}(\mathcal{H},\mathcal{K})}
\def\l{\mathcal{L}}
\def\L{\mathcal{L}}
\def\f{\mathcal{F}}
\def\O{\Omega}
\def\ltwo{{\ell^2}}
\def\set#1#2{\{ #1 \, | \, #2\}}

\def\idk{{\rm id}_{\mathcal{K}}}
\def\idh{{\rm id}_{\mathcal{H}}}
\def\idl{{\rm id}_\ltwo}
\def\idcn{{\rm id}_{\mathbb{C}^n}}
\def\idcm{{\rm id}_{\mathbb{C}^m}}
\def\id#1{{\rm id}_#1}
\def\idd{{\rm id}}
\def\b#1#2{{\rm B}(#1,#2)}

\def\calgm{\mathcal{G}_m}
\def\calrm{\mathcal{R}_m}
\def\calbm{\mathcal{B}_m}
\def\calb{\mathcal{B}}
\def\calr{\mathcal{R}}
\def\calv{\mathcal{V}}
\def\calvh{\mathcal{V}_{\l(\h)}}
\def\calvhc{\mathcal{V}_{\l(\h,\c)}}

\def\calvm{\mathcal{V}_{\l(\M)}}

\def\calvhm{\mathcal{V}_{\l(\h,\M)}}

\def\calvhoj{\mathcal{V}_{\l(\h, \c^J)}}
\def\calg{\mathcal{G}}

\def\vspace{\mathcal{V}}
\def\wkto{\stackrel{\text{wk}}{\to}}
\DeclareMathOperator{\nc}{nc}
\DeclareMathOperator{\lb}{lb}
\DeclareMathOperator{\ran}{ran}
\DeclareMathOperator{\spn}{span}
\DeclareMathOperator{\ball}{ball}
\DeclareMathOperator{\grade}{Grade}
\DeclareMathOperator{\com}{Comm}

\def\hv{{\rm H}(\mathcal{V})}
\def\rv{{\rm R}(\mathcal{V})}
\def\pv{{\rm P}(\mathcal{V})}
\def\cv{{\rm C}(\mathcal{V})}

\def\hvh{{\rm H}_{\l(\h)}(\mathcal{V})}
\def\rvh{{\rm R}_{\l(\h)}(\mathcal{V})}
\def\pvh{{\rm P}_{\l(\h)}(\mathcal{V})}
\def\cvh{{\rm C}_{\l(\h)}(\mathcal{V})}

\def\ctvh{{\rm C}^{\tau}_{\l(\h)}(\mathcal{V})}

\def\htwole{{{\rm  H}_{L,\vare}^2}}
\def\htwolej{{{\rm  H}_{L,\vare}^{2(J)}}}
\def\ltwoi{{\ell^2}^{(I)}}
\def\ltwoj{{\ell^2}^{(J)}}

\def\pd{\mathbb{P}^d}
\newcommand\gdel{G_\delta}
\newcommand\kdel{{K_\delta}}
\def\delp{\delta'}
\newcommand\gdelp{G_{\delta'}}
\def\si{s^{-1}}
\def\elv{\vec{e_l}}
\def\ot{\otimes}

\newcommand\sss{\approx}
\renewcommand\={\ = \ }
\def\vs{\vskip 5pt}
\renewcommand\sin{\Sigma^d_n}
\def\C{\mathbb{C}}
\def\la{\langle}
\def\ra{\rangle}
\def\i{\infty}
\renewcommand\d{\delta}

\def\bd{\begin{defin}}
\def\ed{\end{defin}}
\def\be{\begin{equation}}
\def\ee{\end{equation}}
\def\begm{\left(\begin{matrix}}
\def\endm{\end{matrix}\right)}

\theoremstyle{definition}
\newtheorem{defin}[equation]{Definition}
\newtheorem{lem}[equation]{Lemma}
\newtheorem{prop}[equation]{Proposition}
\newtheorem{thm}[equation]{Theorem}

\newtheorem{fact}[equation]{Fact}
\newtheorem{remark}[equation]{Remark}

\newtheorem{cor}[equation]{Corollary}
\newtheorem{example}[equation]{Example}

\begin{document}

\bibliographystyle{plain}

\maketitle

Abstract: We define a free holomorphic function to be a function that is locally a bounded nc-function.
We prove that free holomorphic functions are the functions that are locally uniformly approximable
by free polynomials. We prove a realization formula and an Oka-Weil theorem for free analytic functions.

\section{Introduction}
\label{seca}

\subsection{Nc-functions and Free holomorphic functions}
\label{subsecaa}

A non-commutative polynomial, also called a free polynomial, in $d$ variables $x^1,\dots, x^d$,
is a finite linear combination of words in the variables, letting the empty word denote the constant $1$. For example
\[
p(x^1,x^2) \=
2 + x^1 -x^1 x^2 x^1 + 3 x^1 x^1 x^2 
\]
is a free polynomial of degree $3$ in 2 variables.
A free poynomial is a natural example of a graded function, which means if one evaluates it on a $d$-tuple
of $n$-by-$n$ matrices, one gets an $n$-by-$n$ matrix. 

Let $\mn$ denote the $n$-by-$n$ matrices over $\C$, and let  $\md$ denote $\cup_{n=1}^\i \mnd$.
%we give $\md$ the disjoint union topology (so a set $U$ is open if and only if $U \cap \mnd$ is open for every $n$).
A graded function is then a map from $\md$
 to $\m := \m^1$ that maps each element in
$\mnd$ to an element in $\mn$. 

Free polynomials have two  further important properties, in addition to being graded:
$p(x \oplus y) = p(x) \oplus p(y)$ and $p(s^{-1} x s) = s^{-1}p( x) s$.
The basic idea of non-commutative function theory is to define a class of graded functions
 that should bear the same relationship to free polynomials as holomorphic functions of $d$ variables
do to commutative polynomials. 

This has been done in a variety of ways: by Taylor \cite{tay73}, in the context of the functional calculus 
for non-commuting operators;
 Voiculescu \cite{voi4,voi10}, in the context of free probability;
Popescu \cite{po06,po08,po10,po11}, in the context of extending classical function theory to 
$d$-tuples of bounded operators;
 Ball, Groenewald and Malakorn \cite{bgm06}, in the context of extending realization formulas
from functions of commuting operators to functions of non-commuting operators;
Alpay and Kalyuzhnyi-Verbovetzkii \cite{akv06} in the context of
realization formulas for rational functions that are $J$-unitary on the boundary of the domain; and
Helton, Klep and McCullough \cite{hkm11a,hkm11b}
and Helton and McCullough \cite{hm12} in the context of developing a descriptive theory of the domains on
which LMI and semi-definite programming apply.
% and Muhly and Solel \cite{}

Very recently, Kaliuzhnyi-Verbovetskyi  and Vinnikov have written a monograph
\cite{kvv12}
that gives a panoramic view of the developments in the field to date.
In their work, functions are defined on nc-domains. Before we say what these are, let us 
establish some notation.  We let 
\begin{eqnarray}
\label{eqar1}
\invn &\ :=\ & \set{M \in \mn}{M \text{ is invertible}}\\
\label{eqar2}
\unin &:=& \set{M \in \mn}{M \text{ is unitary}}.
\end{eqnarray}
For $M_1 = (M_1^1,\ldots,M_1^d) \in \m_{n_1}^d$ and $M_2 = (M_2^1,\ldots,M_2^d) \in \m_{n_2}^d$, we define $M_1 \oplus M_2 \in \m_{n_1+n_2}^d$  by identifying $\c^{n_1} \oplus \c^{n_2}$ with $\c^{n_1 + n_2}$ and direct summing $M_1$ and $M_2$ componentwise, i.e.,
$$M_1 \oplus M_2 = \Big(M_1^1 \oplus M_2^1, \ldots, M_1^d \oplus M_2^d\Big).$$
Likewise, if $M = (M^1,\ldots,M^d)\in \m_n^d$ and $S \in \invn$, we define $S^{-1} M  S\in \m_n^d$ by
$$S^{-1} M S = (S^{-1} M^1 S, \ldots, S^{-1} M^d S).$$

\bd\label{defa1}
If $D \subseteq \m^d$ we say that $D$ is an {\em nc-set} if $D$ is closed with respect to the formation of direct sums and unitary conjugations, i.e.
\be\label{2.20}
\forall_{n_1,n_2}\ \forall_{M_1 \in D \cap \m_{n_1}^d }\  \forall_{ M_2 \in D \cap \m_{n_2}^d}\  M_1 \oplus M_2 \in D \cap \m_{n_1+n_2}^d\notag
\ee
and
\be\label{2.30}
\forall_n\ \forall_{M \in D \cap \mn^d}\ \forall_{U \in \unin}\  U^*MU \in D \cap \mn^d. \notag
\ee

We say that a set $D \subseteq \m^d$ is \emph{nc-open} (resp. \emph{closed, bounded)} if $D \cap M_n^d$ is open (resp. closed,  bounded) for all $n \ge 1$. An \emph{nc-domain} is an nc-set that is nc-open.
\ed

%\begin{defin}\label{defa2}
%An \emph{nc-domain} is an nc-set that is nc-open.
%\end{defin}
\begin{defin}\label{defa3}
An {\em nc-function} is a graded function $\phi$ defined on an nc-domain $D$ such that

i) If $x,y \inn D$, then $\phi(x \oplus y) = \phi(x) \oplus \phi(y)$.

ii) If $s \inn \invn$ and $x, s^{-1} x s \inn D \cap \mnd$ then $\phi(\si x s) = \si \phi(x) s$.

We let $\nc(D)$ denote the set of all nc-functions on $D$.
\end{defin}

In this paper, we shall develop a global theory of holomorphic functions in non-commuting variables, 
by piecing together functions on a nice class of nc-domains,
the basic free open sets.

\begin{defin}\label{defa31}
 if $\delta$ is a matrix of free polynomials in $d$ variables, we define 
\be
\label{eqa3}
G_\delta \= \{ M \in \md: \, \| \delta(M) \| < 1 \} .
\ee
A set of the form \eqref{eqa3} is called a {\em basic free open set}.
The \emph{free topology} on $\m^d$ is the topology that has as a basis
the basic free open sets.
 A \emph{free domain} is a subset of $\m^d$ that is open in the free topology.
\end{defin}
Notice that the intersection of two basic free open sets is another basic free open set,
because $G_{\delta_1} \cap G_{\delta_2} = G_{\delta_1 \oplus \delta_2}$.
Notice also that if $\alpha \inn \C^d$, and we define 
\[
\delta(x) \=
\begm
\frac{1}{\vare}(x^1 - \alpha^1 \idd) \\
\vdots\\
\frac{1}{\vare}(x^d - \alpha^d \idd)
\endm,
\]
then $\gdel \cap \m^d_1$ is the Euclidean ball centered at $\alpha$ of radius $\vare$,
so the free topology agrees with the usual topology on the scalars.

\bd
\label{defa4}
A  free holomorphic function on a free domain $D$ is a  function $\phi$ 
such that every point $M$ in $D$ is contained in a basic free open set $\gdel \subseteq D$
on which $\phi$ 
is  a bounded nc-function.
%satisfies
%
%i) If $x,y, x \oplus y $ are in $D$, then $ \phi(x \oplus y) = \phi(x) \oplus \phi(y)$
%
%ii) If $s$ is invertible in $\mn$ and  $x$ and $s^{-1}xs$ are in  $ D \cap \mnd$, then $\phi(s^{-1}xs) = s^{-1}\phi(x)s$.
\ed

Whereas a basic free open set is an nc-domain, a general free open set may not be,
since it need not be closed under direct sums.
% For example, the 
%free annulus
%\be
%\label{eqay1}
%\bigcup_{0 \leq \theta \leq 2 \pi}
%\{ x \in \m \, : \, \| x -  e^{i\theta} \idd \| < \frac{1}{2} \} 
%\ee
%is not an nc-set.
% When $d =1$, any non-empty free open set is non-empty at
%level $1$, because $M \in \gdel$ implies $\sigma(M) \subset \gdel$.

%We have the following heuristic correspondence between holomorphic functions in several variables and
%free holomorphic functions.
%
%\beq
% { Non-commutative \ setting}&\qquad& { Classical \ Setting}\\
%{\rm Free\ holomorphic \ function} && {\rm Holomorphic \ function}\\
%{\rm Basic\ free\ open\ set} && {\rm Domain\ of\ convergence\ of\ power\ series}\\
%{\rm Free\ }\delta{\rm -realization}&& {\rm Power\ series}\\
%{\rm Free compact\ nc\ set} &&{\rm Compact\ polynomially\ convex\ set} \\
%{\rm Locally\ approximable\ by\ free\ polynomials} &&{\rm Locally\ approximable\ by\ polynomials}\\
%{\rm Theorem~\ref{thm8.10}} && {\rm Oka-Weil\ theorem}\\
%{\rm Theorem~\ref{thmh3}} &&{\rm Corona\ theorem}
%\eeq
%
%The last three correspondences come from the following theorems.
%The utility of this definition is illustrated by the following three theorems.
The locally bounded condition, which one gets automatically in the scalar case, seems to play an
essential r\^ole in   developing an analytic, rather than an  algebraic, theory.
For example, it allows us to give a characterization of free holomorphic functions as
functions that are locally limits of free polynomials.
\vs
{\bf Theorem \ref{thm8.20}.}
Let $D$ be a free domain and let $\phi$ be a graded function defined on $D$. Then $\phi$ is a free holomorphic function if and only if $\phi$ is locally approximable by polynomials.
\vs
A non-commutative power series makes sense, but only when the center is a point in $\m_1^d$.
Given a point $M \in D \cap \mnd$ for some $n \geq 2$, one cannot approximate $\phi$ near $M$ by
expanding a power series about $M$. Being locally approximable by polynomials seems a natural substitute
for analyticity. Rational functions (or, more generally, meromorphic functions built up from
free holomorphic functions) are also free holomorphic, provided one stays away from the poles
(Theorem~\ref{thmj1}).

The classical Oka-Weil theorem states that a holomorphic function on a neighborhood of a compact, polynomially
convex set, can be uniformly approximated by polynomials. See \eg \cite[Chap. 7]{alewer}.
We derive Theorem~\ref{thm8.20} as a special case of a free Oka-Weil theorem.
\vs
{\bf Theorem \ref{thm8.10}.} 
Let $E \subseteq \m^d$ be a  compact set  (in the free topology) that is polynomially convex.
Assume that $\phi$ is a free holomorphic function defined on a neighborhood of $E$. Then $\phi$ can be uniformly approximated by free polynomials on $E$.
\vs
%A further dividend of our approach is a free version of the corona theorem.  
The  corona theorem of Carleson \cite{car62} says that an $N$-tuple of bounded holomorphic functions
on the unit disk is not contained in a proper ideal if and only if the functions are jointly bounded below
by a positive constant. We obtain a free version.

{\bf Theorem \ref{thmh3}.} Let $\{ \psi_i \}_{i=1}^N$ be bounded free holomorphic functions on $\gdel$. Assume for some $\vare > 0$, we have
\[
\sum_{i=1}^N \psi_i(x)^* \psi_i (x) \ \geq \  \vare^2 \, {\rm id}   .
\]
Then there are bounded free holomorphic functions $\phi_i$ on $\gdel$ such that
\[
\sum_{i=1}^N \psi_i(x) \phi_i(x)
\= {\rm id}   .
\]
Moreover, one can choose the functions so that
\[
\| (\phi_1, \dots, \phi_N ) \|\  \leq \ \frac{1}{\vare}.
\]
\vs
Our realization formula Theorem~\ref{thm7.10} can be used to show that every scalar-valued
function on $\gdel$ that is bounded on commuting matrices (using the Taylor functional calculus) can be extended to a free analytic function with the same norm.
\begin{defin}
\label{defmay6}
Let 
$\| f \| _{\delta, {\rm com}} = \sup \{\| f(T) \| \}$, where $T$ ranges over {\em commuting } elements $T$ in $\m_n^d$
that satisfy $\| \delta(T) \| \leq 1$ and $\sigma(T) \subset \gdel$.
Let $H^\i_{\delta, {\rm com}}$ be the Banach
algebra of holomorphic functions on $\gdel$ with this norm. 
\end{defin}

{\bf Theorem \ref{thmh4}.}
Let $$I  \= \{ \phi \in H^\i(\gdel) \, | \, \phi |_{\m^d_1} = 0 \}.
$$
Then 
$H^\i(\gdel)/ I$ is isometrically isomorphic to 
$H^\i_{\delta, {\rm com}}$.

\subsection{The structure of free holomorphic functions}
\label{subseca2}

The engine that drives our results is a model and realization formula for free holomorphic functions
on basic free open sets. To describe these, we must expand the notion of nc-function
 to consider `$\k$-valued' nc-functions on $D$ where $\k$ is a separable Hilbert space. One way to model such objects would be to view them as concrete column vectors with entries in $\nc(D)$. However, we shall adopt an approach which uses tensor products. If $\h$ and $\k$ are Hilbert spaces, we let $\lhk$ denote the bounded linear transformations from $\h$ to $\k$. We identify $(\c^{n_1} \otimes \k) \oplus (\c^{n_2} \otimes \k)$  and $\c^{n_1+n_2} \otimes \k$ in the obvious way. If $T_1 \in \l(\c^{n_1},\c^{n_1} \otimes \k)$ and $T_2 \in \l(\c^{n_2},\c^{n_2} \otimes \k)$, we define $T_1 \oplus T_2 \in \l(\c^{n_1+n_2},\c^{n_1+n_2}\otimes \k)$ by requiring that
$$(T_1 \oplus T_2)(v_1 \oplus v_2) = T_1(v_1) \oplus T_2(v_2 )$$
for all $v_1 \in \c^{n_1}$, $v_2 \in \c^{n_2}$, and $k \in \k$.

\bd
\label{defa5}
We say a function $f$ is a $\k$-valued nc-function if the domain of $f$ is some nc-domain, $D$,
\be\label{2.80}
\forall_{n}\  \forall_{x \in D \cap \mn^d}\  f(x) \in \l(\cn,\cn \otimes \k),
\ee
\be\label{2.90}
\forall_{x,y \in D}\ f(x \oplus y) = f(x) \oplus f(y), \text{ and}
\ee
\be\label{2.100}
\forall_{n}\ \forall_{x \in D \cap \mn^d}\  \forall_{s \in \invn}\ s^{-1} x s \in D \implies f(s^{-1}x s) = (s^{-1}\otimes \idk)f(x)s.
\ee
If $D$ is an nc-domain, we let $\nc_\k(D)$ denote the collection of $\k$-valued  nc-functions on $D$.
\ed

Let $p$ be a free polynomial, and $f$ be in $\nc_\k(D)$. Then we define $pf \in \nc_\k(D)$
by
\[
pf(x) \= [p(x) \otimes \idk] f(x) .
\]
Now let $\delta$ be an $I$-by-$J$ matrix of free polynomials, and let $u$ be in $\nc_{{\ell^2}^{(J)}}(D)$. We define 
$
\delta u \in  \nc_{{\ltwo}^{(I)}}(D)$ by matrix multiplication. Let $u = (u_1, \dots, u_J)^t$;
then define 
$\delta u$ by  the formula
\[
(\delta u)(x) \= 
\begm
\sum_{j=1}^J  [\delta_{1j}(x) \otimes \idl ] u_j(x) \\
\vdots \\
\sum_{j=1}^J  [\delta_{Ij}(x) \otimes \idl ] u_j(x)
\endm
 \qquad  x \in D.
\]
\begin{defin}\label{defa6}
Let $\phi$ be a graded function on $G_\delta$. A \emph{$\delta$ nc-model for $\phi$}
% (standard)}
 is a formula of the form
\be\label{7.10x}
1-\phi(y)^*\phi(x) = u(y)^*[ 1-\delta(y)^*\delta(x) ]u(x), \qquad x,y \in G_\delta
\ee
where $u$ is in $\nc_{{\ell^2}^{(J)}}G_\delta$. 
\end{defin}
\begin{defin}\label{defa7}
Let $\phi$ be a graded function on $G_\delta$.  A \emph{free $\delta$-realization for $\phi$ }
 is an isometry
\[
\mathcal{J}\= \begin{bmatrix}A&B\\C&D\end{bmatrix} 
\]
 such that for each $n\in \n$ and each $x \in \gdel \cap \mn^d$
\be
\label{may6_2}
\phi(x) = (\id{\cn} \ot A) +  (\id{\cn} \ot B)\delta(x) [\idd - ( \id{\cn}\ot D) \delta(x)]^{-1}
(\id{\cn} \ot C) .
\ee
\end{defin}
%Having a $\delta$-model is equivalent to having a free $\delta$-realization.
We prove in Theorem~\ref{thm7.10} that  every free holomorphic function that is bounded in norm by $1$ on $\gdel$ has a 
$\delta$-model and a free  $\delta$-realization. 

In the commutative case, and when $D$ is the polydisk, the result was first proved in \cite{ag90}.
The extension to $\gdel$ for scalar valued functions was first done by Ambrozie and Timotin \cite{at03}; Ball and Bolotnikov extended this result to functions of commuting operators in  \cite{babo04}. In the non-commutative case, the first
version of this result
%,  for formal power series rather than nc-functions,
was proved by Ball, Groenewald and Malakorn \cite{bgm06}.
They proved a realization formula for non-commutative power series
%Their results
on  domains that could be described in terms of certain bipartite graphs; these include
the most important examples, the non-commutative polydisk and the non-commutative ball.

The  statement of the theorem is as follows (we omit Statement (2) for now).
We extend the notion of nc function to an $\l(\h,\k)$-valued function in the natural way
(see Definition~\ref{defax1}).

{\bf Theorem \ref{thm7.10}} Let $\h,\k_1,\k_2$ be finite dimensional Hilbert spaces.
 Let $\delta$ be an $I\times J$ matrix whose entries are free polynomials.
Let $\Psi$ be a graded $\l(\h,\k_1)$-valued function on $\gdel$,
and let  $\Phi$ be a graded $\l(\h,\k_2)$-valued  function on $G_\delta$. 
%Let $\Theta(y,x) = \Psi(y)^* \Psi(x) - \Phi(y)^* \Phi(x)$.
The following are equivalent.\\
\qquad (1) $\Psi(x)^* \Psi(x) - \Phi(x)^* \Phi(x) \geq 0$ on $\gdel$.\\
%\qquad (2) $\Theta$ has a $\delta$ nc-model.\\
\qquad (3) There exists an nc $\l(\k_1,\k_2)$-valued function $\O$  satsifying
$\O \Psi = \Phi$ and such that $\O$ has a  free $\delta$-realization.
\vs
In the special case that $\Psi$ is the identity, this says that every bounded free analytic function has a 
free $\delta$-realization as in \eqref{may6_2}.

\section{Structure of the Paper}
\label{secb}

In Section~\ref{secbn} we discuss basic notions of nc domains and nc functions.
We prove that every nc function on a domain $D$ extends  to an nc function on
its envelope $D^\sim$, the similarity closed set generated by $D$ (Proposition~\ref{prop2.40}).

In Section~\ref{seclb}, we prove that locally bounded nc functions are holomorphic
(Theorem \ref{thm3.10}).
We define a free holomorphic function to be a locally bounded nc function, and prove that
Montel's theorem holds for these functions (Proposition \ref{prop3.10}).

To prove that bounded free holomorphic functions have realizations, we use a Hahn-Banach argument. To make this work, we need to know that the set of all functions of the form
\[
 u(y)^*[ 1-\delta(y)^*\delta(x) ]u(x), \qquad  u \in \nc_{{\ell^2}^{(J)}}G_\delta
\]
is a closed cone. Proving it is closed is delicate, so we rely on finite dimensional approximations.
In Section~\ref{secpncs} we develop the theory of partial nc-sets and partial nc-functions,
which are restrictions to finite sets of nc-functions. To allow us to piece these together
into an nc-function, we introduce the notion of a well-organized pair $(E,\s)$
(Definition \ref{def4.10}), which is a finite set $E$ and a finite number of similarities
with certain nice properties.

In Section \ref{secwo}, we show how to get $\delta$-models and $\delta$-realizations
on well-organized pairs. In Section~\ref{secfm}, we piece these together to get
a  $\delta$ nc-model on the whole set $\gdel$. The main theorem here is Theorem~\ref{thm6.10}.
We improve this theorem in Section~\ref{ssecdnc} to get Theorem~\ref{thm7.10},
which says one can find a free $\delta$-realization for the multiplier $\O$ .

In Section~\ref{seci1}  we use this structure theorem to derive our major consequences: the free Oka-Weil Theorem \ref{thm8.10}, which in particular gives 
a proof that  a function is free holomorphic if and only if it is locally approximable by free polynomials (Theorem~\ref{thm8.20}).

In Section~\ref{secj}, we prove that free meromorphic functions are free holomorphic off their singular sets. We give an index to notation and definitions in Section~\ref{secind}.

\section {Basic Notions}
\label{secbn}

\subsection{nc-bounded}
\label{ssecncd}

We define the nc-norm $\norm{\cdot}$ on each set $\mn^d$ by the formula
 $$\norm{M} = \max_{1 \le r \le d}\norm{M^r}$$
and when metric calculations are required, we shall use the nc-metric $d$, defined on each set $\mn^d$ by the formula
$$d(M,N)=\max_{1 \le r \le d}\norm{M^r-N^r}.$$
If $M \in D \cap \mn^d$ and $r >0$, we let
$$\b{M}{r} = \set{N \in \mn^d}{d(M,N) <r}. $$
Evidently, a set $D \subseteq \m^d$ is nc-bounded when
$$\sup_{M \in D \cap \mn^d} \norm{M} < \infty$$
for each $n \ge 1$. We say that a set $D \subseteq \m^d$ is \emph{bounded} if
$$\sup_{M \in D} \norm{M} < \infty.$$
Clearly, boundedness implies nc-boundedness but not conversely.

\subsection{Envelopes of nc-Domains}
\label{ssecen}

If $A \subseteq \m^d$, let us agree to say that $A$ is \emph{invariant} if for each $n \ge 1$ and each $S \in \invn$,
\be\label{2.40}
S^{-1}(A \cap \mn^d)S \subseteq A \cap \mn^d. \notag
\ee
As the intersection of invariant nc-sets is an invariant nc-set, it is clear that if $A \subseteq \m^d$, then there exists a \emph{smallest} invariant nc-set containing $A$. We formalize this fact in the following definition.
\begin{defin}\label{def2.20}
If $A \subseteq \m^d$, then $A^\sim$, the envelope of $A$, is the unique invariant nc-set satisfying $A \subseteq A^\sim$ and $A^\sim \subseteq B$ whenever $B$ is an invariant nc-set containing $A$.
\end{defin}
\begin{prop}\label{prop2.10}
Let $A \subseteq \m^d$ and let $M \in \mn^d$. $M \in A^\sim$ if and only if there exist an integer $m\ge 1$, integers $n_1,\ n_2,\ \ldots,\ n_m \ge 1$ satisfying $n=n_1+n_2+\ldots+n_m$, matrix tuples $M_1 \in A \cap \m_{n_1}^d, M_2 \in A \cap \m_{n_2}^d,\ \ldots,\ M_m \in A \cap \m_{n_m}^d$, and $S \in \invn$ such that
\be\label{2.x80}
M=S^{-1}(\bigoplus_{k=1}^m M_k)S
\ee
\end{prop}
\begin{proof}
Let $B$ denote the collection of matrix tuples $M$ that have the form as presented in \eqref{2.x80}. Then $B$ is an invariant nc-set. Also, $B \subseteq C$ if $C$ is an invariant nc-set that contains $A$. Therefore, $B=A^\sim$.
\end{proof}
As corollaries to Proposition \ref{prop2.10} we obtain the following two facts which will prove useful in the sequel.
\begin{prop}\label{prop2.20}
If $A$ is an nc-set and $N \in \mn^d$, then $N \in A^\sim$ if and only if there exists $M \in A \cap \mn^d$ and $S \in \invn$ such that $N = S^{-1}MS$.
\end{prop}
\begin{prop}\label{prop2.30}
If $D$ is an nc-domain, then  $D^\sim$ is an nc-domain.
\end{prop}
\begin{proof}
Let $D \subseteq \m^d$ be an nc domain. Fix $n\ge 1$ and $N \in D^\sim \cap \mn^d$. By Proposition \ref{prop2.20} there exist $M \in D \cap \mn^d$ and $S \in \invn$ such that $N = S^{-1}MS$. As $D \cap \mn^d$ is open, there exists $\delta > 0$ such that
$$M + S\Delta S^{-1} \in D \cap \mn^d$$
whenever  $\Delta \in \mn^d$ and $\norm{\Delta} < \delta$. Consequently, if $\Delta \in \mn^d$ and $\norm{\Delta} < \delta$, then
$$N + \Delta = S^{-1} M S + \Delta = S^{-1}(M+S \Delta S^{-1}) S \in D^\sim.$$
\end{proof}

\subsection{nc-Functions}
\label{ssecncf}

We defined nc-functions and $\k$-valued nc-functions in Definitions~\ref{defa3} and \ref{defa5}.
We extend this to `$\lhk$-valued' nc-functions on $D$ where $\h$ and $\k$ are Hilbert spaces. If $T_1 \in \l(\c^{n_1} \otimes \h,\c^{n_1} \otimes \k)$ and $T_2 \in \l(\c^{n_2} \otimes \h,\c^{n_2} \otimes \k)$ we define $T_1 \oplus T_2 \in \l(\c^{n_1+n_2} \otimes \h,\c^{n_1+n_2} \otimes \k)$ by requiring that
$$(T_1 \oplus T_2)((v_1 \oplus v_2) \otimes h) = T_1(v_1 \otimes h) \oplus T_2(v_2 \otimes h)$$
for all $v_1 \in \c^{n_1}$, $v_2 \in \c^{n_2}$, and $h \in \h$.
\begin{defin}
\label{defax1}
We say a function $f$ is an $\lhk$-valued nc-function 
(and write $f \, \in \, {\rm nc}_{\lhk}$ )
if the domain of $f$ is some nc-domain, $D$,
\be\label{2.110}
\forall_{n}\ \forall_{ x \in D \cap \mn^d}\  f(x) \in \l(\cn \otimes \h,\cn \otimes \k),
\ee
\be\label{2.120}
\forall_{x,y \in D}\ f(x \oplus y) = f(x) \oplus f(y), \text{ and}
\ee
\be\label{2.130}
\forall_{n}\ \forall_{x \in D \cap \mn^d}\ \forall_{s \in \invn}\ s^{-1} x s \in D \implies f(s^{-1}x s) = (s^{-1}\otimes \idk)f(x)( s \otimes \idh).
\ee
\end{defin}
A simple yet important point is that if $\dim(\h) = \dim(\k) =1$, then we can identify $\l(\cn \otimes \h,\cn \otimes \k)$ with $\mn$ and with this identification it is easy to verify that \eqref{2.110}, \eqref{2.120}, and \eqref{2.130} imply that Definition~\ref{defa3} is satisfied.
 Thus, theorems proved for $\lhk$-valued nc-functions hold for nc-functions. Likewise, theorems proved for $\lhk$-valued nc-functions hold for $\k$-valued nc-functions.

The following Proposition is also proved in \cite{bkvv12}.
\begin{prop}\label{prop2.40}
Let $\h$ and $\k$ be Hilbert spaces. If $D$ is an nc-domain and $f$ is an $\lhk$-valued nc-function on $D$, then there exists a unique nc-function $f^\sim$ on $D^\sim$ such that $f^\sim | D = f$.
\end{prop}
\begin{proof}
Fix $N \in D^\sim \cap \mn^d$. By Proposition \ref{prop2.20} there exists $M \in D \cap \mn^d$ and invertible $s\in \mn$ such that $N = s^{-1}Ms$. We define
\be\label{2.140}
f^\sim(N) = (s^{-1} \otimes \idk)f(M)(s \otimes \idh)
\ee
We need to prove two things: that $f^\sim$ is well defined, and that $f^\sim$ is an $\lhk$-valued nc-function.

To see that $f^\sim$ is well defined, fix $N \in D^\sim \cap \mn^d$ and then choose $M_1,M_2 \in D \cap \mn^d$ and invertible $s_1,s_2 \in \mn$ with $s_1^{-1}M_1s_1 = N$ and $ s_2^{-1} M_2s_2 =N.$ If we set $s=s_1s_2^{-1}$, then as $s^{-1}M_1s = M_2 \in D$, it follows from \eqref{2.130}, that
$$ f(s^{-1}M_1 s) = (s^{-1}\otimes \idk)f(M_1)( s \otimes \idh).$$
Hence,
\beq
(s_1^{-1} \otimes \idk)f(M_1)(s_1 \otimes \idh) &\=&
(s_2^{-1} \otimes \idk)(s^{-1} \otimes \idk)f(M_1)(s \otimes \idh)(s_2 \otimes \idh)\\ 
&=&(s_2^{-1} \otimes \idk)f(M_2)(s_2 \otimes \idh).
\eeq
This proves that $f^\sim$ is well defined.

To see that $f^\sim$ is an $\lhk$-valued nc-function on $D^\sim$, note first that \eqref{2.110} follows immediately from \eqref{2.140}. To prove \eqref{2.120} fix $N_1 \in D^\sim \cap \m_{n_1}^d$ and $N_2 \in D^\sim \cap \m_{n_2}^d$. Choose $M_1 \in D \cap \m_{n_1}^d$, $N_2 \in D^\sim \cap \m_{n_2}^d$, $s_1 \in \mathcal{I}_{n_1}$, and $s_2 \in \mathcal{I}_{n_2}$ such that $N_1 = s_1^{-1}M_1s_1$ and $N_2 = s_2^{-1}M_2s_2$. Then, as
$$N_1 \oplus N_2 = (s_1 \oplus s_2)^{-1}(M_1 \oplus M_2)(s_1 \oplus s_2),$$
 and $M_1 \oplus M_2 \in D$, we have using \eqref{2.140} that
\begin{align*}
f^\sim(N_1 \oplus N_2) &=\big((s_1 \oplus s_2)^{-1} \otimes \idk \big)f(M_1 \oplus M_2)\big((s_1 \oplus s_2) \otimes \idh \big)\\ 
&=\big((s_1^{-1} \otimes \idk) \oplus  (s_2^{-1} \otimes \idk)\big)
\big(f(M_1) \oplus f(M_2)\big)\big((s_1 \otimes \idh) \oplus  (s_2 \otimes \idh)\big)\\ 
&=\big((s_1^{-1} \otimes \idk) f(M_1)(s_1 \otimes \idh)\big) \oplus
\big((s_2^{-1} \otimes \idk) f(M_2)(s_2 \otimes \idh)\big)\\ 
&=f^\sim(N_1) \oplus f^\sim(N_2).
\end{align*}
This proves \eqref{2.120}.

Finally, to prove \eqref{2.130}, fix $N \in D^\sim \cap \mn^d$ and $s \in \invn$. Choose $M \in D \cap \mn^d$ and $t \in \invn$ such that $N = t^{-1}Mt$. Then, as $s^{-1}Ns = (ts)^{-1}M(ts)$,
\begin{align*}
f^\sim(s^{-1}Ns)&= ((ts)^{-1} \otimes \idk)f(M)((ts) \otimes \idh)\\
&=(s^{-1} \otimes \idk)\big((t^{-1} \otimes \idk)f(M)(t \otimes \idh)\big)(s \otimes \idh)\\
&=(s^{-1} \otimes \idk)f^\sim(N)(s \otimes \idh).
\end{align*}
This proves \eqref{2.130}.
\end{proof}

More generally, when $D \subseteq \m^d$ is an nc-domain and $f \in \nc_{\lhk}(D)$, it is possible to extend $f$ in the following way. If $\vspace$ is an $n$-dimensional vector space, $T$ is a d-tuple of linear transformations on $\vspace$ and there exists an invertible linear map $S:\vspace \to \c^n$ such that
\[
STS^{-1} = (ST^1S^{-1},\ldots,ST^d S^{-1}) \in D \cap \mn^d,
\]
then define $f^{\sss}:\vspace \otimes \h \to \vspace \otimes \k$ by the formula,
\be
\label{bx1}
f^{\sss}(T) = (S^{-1}\otimes \idk) f(STS^{-1})(S\otimes \idh).
\ee
 It is straightforward to check that with this definition $f^{\sss}$ is well defined on $D^{\sss}$, the set of all linear transformations on finite dimensional vector spaces that are similar to an element of $D$, and that the appropriate analogs of \eqref{2.110}, \eqref{2.120}, and \eqref{2.130} hold.

Note to the reader: if $f \in \nc_{\lhk}(D)$, we can apply $f$ to $d$-tuples of matrices on $\cn$;
we can apply $f^{\sss}$ to $d$-tuples of linear transformations on any
 finite dimensional vector space.

We close this section with the following useful lemmas. Both are simple modifications of results from \cite{hkm11b}.
\begin{lem}\label{lem2.20}(cf. Lemma 2.6 in \cite{hkm11b}). 
Let $D$ be an nc-domain in $\m^d$, let $\h$ and $\k$ be Hilbert spaces, and let $f$ be an $\lhk$-valued nc-function on $D$. Fix $n \ge 1$ and $C \in \m_n$. If $M,N \in D \cap \mn^d$ and
\be\label{2.150}
\begin{bmatrix}N&NC-CM\\ 0&M\end{bmatrix} \in D \cap \m_{2n}^d,
\ee
then
\be\label{2.160}
f(\begin{bmatrix}N&NC-CM\\ 0&M\end{bmatrix}) = \begin{bmatrix}f(N)&f(N)C-Cf(M)\\ 0&f(M)\end{bmatrix}.
\ee
\end{lem}
\begin{proof}
Let
$$s = \begin{bmatrix}\idcn&C\\0&\idcn\end{bmatrix}$$
so that
$$\begin{bmatrix}N&NC-CM\\ 0&M\end{bmatrix} = s^{-1}\begin{bmatrix}N&0\\ 0&M\end{bmatrix}s.$$
Using \eqref{2.120} and \eqref{2.130},
\begin{align*}
\lefteqn{
f(\begin{bmatrix}N&NC-CM\\ 0&M\end{bmatrix}) = f(s^{-1}(N \oplus M)s)}\\ \\
&=(s^{-1}\otimes \idk)(f(N) \oplus f(M))(s \otimes \idh)\\ \\
&=\begin{bmatrix}\idcn \otimes \idk&-C \otimes \idk\\0&\idcn \otimes \idk\end{bmatrix}\begin{bmatrix}f(N)&0\\ 0&f(M)\end{bmatrix} \begin{bmatrix}\idcn \otimes \idh&C \otimes \idh\\0&\idcn \otimes \idh\end{bmatrix}\\ \\
&=\begin{bmatrix}f(N)&f(N)C-Cf(M)\\ 0&f(M)\end{bmatrix}.
\end{align*}
\end{proof}
\begin{lem}\label{lem2.30}(cf. Proposition 2.2 in \cite{hkm11b}). 
Let $D$ be an nc-domain, let $\h$ and $\k$ be Hilbert spaces, and let $f$ be an $\lhk$-valued nc-function on $D$. Let $\calv$ and $\mathcal{W}$ be vector spaces, and let $R:\calv \to \calv$, $T:\mathcal{W} \to \mathcal{W}$, and $L:\calv \to \mathcal{W}$ be linear transformations. If $R,T \in D^{\sss}$ and 
\[
TL = LR,
\]
then
\[
f^{\sss} (T) ( L \otimes \idh) \= (L \otimes \idk) f^{\sss} (R).
\] 
\end{lem}  
\begin{proof}
Let 
$\displaystyle s = \begin{bmatrix}\id{\mathcal{W}}&L\\0&\id{\mathcal{V}}\end{bmatrix}$
and use
\[
f^{\sss} ( s^{-1} 
 \begin{bmatrix}T&0\\0&R\end{bmatrix}
s ) \=
s^{-1} f^{\sss} (
 \begin{bmatrix}T&0\\0&R\end{bmatrix} )
s .
\]
\end{proof}

\section{Local Boundedness and Holomorphicity}
\label{seclb}

In this section we shall prove that locally bounded nc-functions are automatically holomorphic. In addition we shall lay out various tools involving locally bounded and holomorphic graded functions (not necessarily assumed to be nc-functions) that will be heavily used in the sequel.
Most of the content of this section also appears in \cite[Chapter 7]{kvv12}.

If $D$ is an nc-domain in $\m^d$ and $\h$ and $\k$ are Hilbert spaces, then we say a function $f$ defined on $D$ is a \emph{graded $\lhk$-valued function on $D$} if
\be\label{3.10}
\forall_n\ \forall_{x \in D \cap \mn^d}\ f(x) \in \l(\c^n \otimes \h,\c^n \otimes \k).
\ee
\begin{defin}\label{def3.10}
Let $D$ be an nc-domain in $\m^d$ and let $\h$ and $\k$ be Hilbert spaces. We say that a graded $\lhk$-valued function on $D$ is locally bounded if for each $n \ge 1$ and each $x \in D \cap \mn^d$, there exists $r>0$ such that $\b{x}{r} \subseteq D$ and
$$\sup_{y \in \b{x}{r}}\norm{f(x)} < \infty.$$
If $\f$ is a collection of graded $\lhk$-valued functions on $D$, we say that $\f$ is locally uniformly bounded if for each $n \ge 1$ and each $x \in D \cap \mn^d$, there exists $r>0$ such that $\b{x}{r} \subseteq D$ and
$$\sup_{y \in \b{x}{r}}\sup_{f \in \f}\norm{f(x)} < \infty.$$
\end{defin}
\begin{prop}\label{prop3.x20}
Let $D$ be an nc-domain in $\m^d$, let $\h$ and $\k$ be Hilbert spaces, and let $f$ be an $\lhk$-valued nc-function on $D$. If $f$ is locally bounded on $D   $, then $f^\sim$ is locally bounded on $D^\sim$. If $\f$ is a locally uniformly bounded collection of graded $\lhk$-valued functions on $D$, then $\f^\sim$ is a locally uniformly bounded collection of graded $\lhk$-valued functions on $D^\sim$.
\end{prop}
We view $\m^d = \cup_n \mn^d$ as being endowed with the disjoint union topology, i.e., $G \subseteq \m^d$ is open if and only if $G\cap \mn^d$ is open for each $n \ge 1$.  If $K \subseteq \m^d$ is a compact set in this topology, then as $\mn^d$ is open for each $n \ge 1$ and $K \subseteq \cup_n \mn^d$, it follows that there exists $n \ge 1$ such that
\[
K \subseteq \bigcup_{m=1}^n \m_{m}^d.
\]

Fix an nc-domain $D \subseteq \m^d$. By a \emph{compact-open exhaustion of $D$} we mean a sequence of compact subsets of $D$, $\langle K_m \rangle$, satisfying $K_m \subseteq K_{m+1}^\circ$ for all $m \ge 1$ and such that
\[
D = \bigcup_{m=1}^\infty K_m.
\]
A particularly simple way to construct a compact-open exhaustion of $D$ is to note that as $D \cap \mn^d$ is an open subset of $\mn^d$ for each $n \ge 1$, for each $n$ there exists a compact-open exhaustion, $\langle K_{n\, m} \rangle$,  of $D \cap \mn^d$. It follows that if $K_m$ is defined by
\[
K_m = \bigcup_{n=1}^m K_{n \, m},
\]
then $\langle K_m \rangle$ is a compact-open exhaustion of $D$.
In the sequel notions introduced using a compact-open exhaustion of $D$ can in each case shown to be independent of the particular choice of exhaustion. Also, for convenience we assume that the exhaustion has been chosen to satisfy the property that
\[
\forall_m\ K_m \subseteq \bigcup_{n=1}^m D \cap \mn^d.
\]

 Now let $D \subseteq \m^d$ be an nc-domain and let $\langle K_m \rangle$ be a compact-open exhaustion of $D$. If $\lb_{\lhk}(D)$ denotes the space of locally bounded graded $\lhk$-valued functions on $D$, then for $f \in \lb_{\lhk}(D)$,
 \[
 \rho_m(f) \stackrel{\text{def}}{=}\sup_{x \in K_m} \norm{f(x)} < \infty
 \]
 for each $m\ge 1$. It follows that $d:\lb_{\lhk}(D) \times \lb_{\lhk}(D) \to \r$ defined by
 \be\label{3.15}
 d(f,g) = \sum_{m=1}^\infty 2^{-m} \frac{\rho_m(f-g)}{1+ \rho_m(f-g)}
 \ee
 is a translation invariant metric on $\lb_{\lhk}(D)$. If $f \in \lb_{\lhk}(D)$ and $\langle f^(k)\rangle$ is a sequence in $\lb_{\lhk}(D)$ we shall write $f^(k) \to f$ if $d(f,f^{(k)}) \to 0$.
\begin{defin}\label{def3.20}
Let $D$ be an nc-domain in $\m^d$ and let $\h$ and $\k$ be Hilbert spaces. Let $f$ be a graded $\lhk$-valued function on $D$. We say that $f$ is holomorphic on D if for each $n \ge 1$, $x \to f(x)$ is an holomorphic $\lhk$-valued function in the entries of $x$.
\end{defin}
An important tool (Proposition \ref{prop3.30} below) that we shall use frequently in the sequel is based on the application of Montel's Theorem to uniformly locally bounded sequences of graded holomorphic $\lhk$-valued nc-functions. Unfortunately, in the cases when either $\h$ or $\k$ is infinite dimensional, the topology induced by the metric defined in \eqref{3.15} is too strong for this purpose. Accordingly, we define the following notion of \emph{weak convergence}.

Let $\langle K_m\rangle$ be a compact-open exhaustion of an nc-domain $D$  as above. If $\langle f^{(k)}\rangle$ is a sequence of graded $\lhk$-valued functions on $D$ and $f$ is a graded $\lhk$-valued function on $D$, we say that $f^{(k)} \stackrel{\text{wk}}{\to} f$ if for each $m,n \ge 1$ such that $K_m \cap \mn^d \ne \emptyset$, for each $c,d \in \cn$, and for each $h \in \h$ and $k \in \k$ we have that
\[
\lim_{k \to \infty}\ \sup_{x \in K_m\cap \mn^d}\  \langle (f^{(k)}(x) - f(x))c\otimes h ,d\otimes k\rangle = 0.
\]
\begin{prop}\label{prop3.30}
Let $D$ be an nc-domain and let $\langle f^{(k)}\rangle$ be a uniformly locally bounded sequence of graded holomorphic $\lhk$-valued functions on $D$. Then there exists a subsequence $\langle f^{(k_j)}\rangle$ and a graded holomorphic $\lhk$-valued function $f$ on $D$ such that $ f^{(k_j)} \wkto f$.
\end{prop}
\begin{proof}
The proof will proceed by doing a diagonal subsequence argument twice. First fix $m$ and $n$ such that $K_m \cap \mn^d \ne \emptyset$. Let $\{e_i\}$ denote the standard orthonormal basis for $\cn$ and fix orthonormal bases $\{h_l\}$ and $\{k_l\}$ for $\h$ and $\k$. For each $i_1,i_2 \le n$ and each $l_1$ and $l_2$,
\[
\langle f^{(k)}(x)\ c_{i_1} \otimes h_{l_1}, c_{i_2}\otimes k_{l_2}\rangle
\]
is uniformly bounded on a neighborhood of $K_m \cap \mn^d$. Therefore using Montel's Theorem and mathematical induction, for each $N\ge 1$, there exist an increasing sequence of integers $\langle k_{N,j}\rangle$ and holomorphic functions $g_{i_1,l_1,i_2,l_2}^N$ defined on a neighborhood of $K_m \cap \mn^d$ such that
\[
\langle f^{(k_{N,j})}(x)\ c_{i_1} \otimes h_{l_1}, c_{i_2}\otimes k_{l_2}\rangle \to g_{i_1,l_1,i_2,l_2}^N(x)
\]
 uniformly on a neighborhood of $K_m \cap \mn^d$ for all $i_1,i_2 \le n$ and $l_1,l_2 \le N$ and with the additional property that $\langle k_{N+1,j}\rangle$ is a subsequence of $\langle k_{N,j}\rangle$ for each $N$. Hence, there exist holomorphic functions $g_{i_1,l_1,i_2,l_2}$ defined on a neighborhood of $K_m \cap \mn^d$ such that
 \[
\langle f^{(k_{N,N})}(x)\ c_{i_1} \otimes h_{l_1}, c_{i_2}\otimes k_{l_2}\rangle \to g_{i_1,l_1,i_2,l_2}(x)
 \]
uniformly on a neighborhood of $K_m \cap \mn^d$ for all $i_1,i_2 \le n$ and $l_1,l_2$.

Summarizing, we have proved the following fact.
\begin{fact}\label{fact3.10}
that if $\langle f^{(k)}\rangle$ is a uniformly bounded sequence of graded holomorphic $\lhk$-valued functions on $D$, then for each $m$ and $n$ such that $K_m \cap \mn^d \ne \emptyset$, there exist a strictly increasing sequence $\langle k_N\rangle$ and holomorphic functions  $g_{i_1,l_1,i_2,l_2}$ defined on a neighborhood of $K_m \cap \mn^d$ such that
 \[
\langle f^{(k_N)}(x)\ c_{i_1} \otimes h_{l_1}, c_{i_2}\otimes k_{l_2}\rangle \to g_{i_1,l_1,i_2,l_2}(x)
 \]
  uniformly on a neighborhood of $K_m \cap \mn^d$ for all $i_1,i_2 \le n$ and $l_1,l_2$.
 \end{fact}
Now fix $n$. For $m \ge n$ we use Fact \ref{fact3.10} to inductively construct an increasing sequence $\langle k_{m,N}\rangle$ and holomorphic functions $g_{i_1,l_1,i_2,l_2}^m$ defined on a neighborhood of $K_m \cap \mn^d$ satisfying
 \be\label{3.20}
\langle f^{(k_{m,N})}(x)\ c_{i_1} \otimes h_{l_1}, c_{i_2}\otimes k_{l_2}\rangle \to g_{i_1,l_1,i_2,l_2}^m(x)
 \ee
 uniformly on a neighborhood of $K_m \cap \mn^d$ for all $i_1,i_2 \le n$ and $l_1,l_2$ and with $\langle k_{m,N}\rangle$ a subsequence of $\langle k_{m+1,N}\rangle$ for each $m$. As $K_m \subseteq K_{m+1}$, it follows from \eqref{3.20} that if $m_1 \le m_2$, then
\[
g_{i_1,l_1,i_2,l_2}^{m_1}(x) = g_{i_1,l_1,i_2,l_2}^{m_2}(x)
\]
on a neighborhood of $K_{m_1} \cap \mn^d$. Therefore, as $\{K_m\}$ is an exhaustion of $D$, we can define an holomorphic function $ g_{i_1,l_1,i_2,l_2}^n:D \cap \mn^d \to \mn$ by the formula
 \be\label{3.22}
 g_{i_1,l_1,i_2,l_2}^n(x) = g_{i_1,l_1,i_2,l_2}^{m}(x) \text{ if } m\ge n \text{ and } x \in D \cap \mn^d.
 \ee
 Now define a graded holomorphic $\lhk$-valued function $f$ on $D$ by requiring that
 \[
 \langle f(x)\ c_{i_1} \otimes h_{l_1}, c_{i_2}\otimes k_{l_2}\rangle =  g_{i_1,l_1,i_2,l_2}^n(x)
 \]
whenever $n \ge 1$, $x \in D \cap \mn^d$, $i_1,i_2 \le n$, $l_1,l_2 \ge1$. By \eqref{3.20} and \eqref{3.22} it follows that
\[
 \langle f^{(k_{m,m})}(x)\ c_{i_1} \otimes h_{l_1}, c_{i_2}\otimes k_{l_2}\rangle \to \langle f(x)\ c_{i_1} \otimes h_{l_1}, c_{i_2}\otimes k_{l_2}\rangle
\]
whenever $n \ge 1$, $x \in D \cap \mn^d$, $i_1,i_2 \le n$, $l_1,l_2 \ge1$. Hence, since  $\langle f^{(k)}\rangle$ is assumed locally bounded it follows that $f(x) \in \l(\c^n\otimes \h,\cn \otimes \k)$ for all $x \in D \cap \mn^d$ and $f^{k_{m,m}} \wkto f$.
\end{proof}
\begin{thm}\label{thm3.10}
Let $D$ be an nc-domain in $\m^d$, let $\h$ and $\k$ be Hilbert spaces, and let $f$ be an $\lhk$-valued nc-function on $D$. If $f$ is locally bounded on $D$, then $f$ is holomorphic on $D$.
\end{thm}
\begin{proof}
The proof will proceed in two steps. We first show that if $f$ is locally bounded, then $f$ is continuous. That $f$ is holomorphic will then follow by a straightforward modification of Proposition 2.5 in \cite{hkm11b}.

Fix $M \in D \cap \mn^d$ and let $\epsilon > 0$. Choose $r>0$ so that
\[
\b{\begin{bmatrix}M&0\\0&M\end{bmatrix}}{r} \subseteq D \cap \m_{2n}^d.
\]
If  $s$ is chosen with $0<s<r$ then as $\b{M\oplus M}{s}^-$ is a compact subset of $D \cap \m_{2n}^d$ and $f$ is assumed locally bounded, there exists a constant $B$ such that
\be\label{3.30}
x\in \b{\begin{bmatrix}M&0\\0&M\end{bmatrix}}{s} \implies \norm{f(x)} < B.
\ee
Choose $\delta$ sufficiently small so that $\delta< \min\{s\epsilon/sB,s/2\}$ and $\b{M}{\delta} \subseteq D$. That $f$ is continuous at $M$ follows from the following claim.
\be\label{3.40}
N \in \b{M}{\delta} \implies f(N) \in \b{f(M)}{\epsilon}.
\ee
To prove the claim fix $N \in \mn^d$ with $\norm{N-M} < \delta$. Then $\norm{N-M} <s/2$ and $\norm{(B/\epsilon)(N-M)} <s/2$. Hence by the triangle inequality,
\[
\norm{\begin{bmatrix}N&c(N-M)\\ 0&M\end{bmatrix}-\begin{bmatrix}M&0\\ 0&M\end{bmatrix} }< s.
\]
Hence, by \eqref{3.30},
\[
\norm{f(\begin{bmatrix}N&(B/\epsilon)(N-M)\\ 0&M\end{bmatrix})}< B.
\]
But $M$, $N$, and $\begin{bmatrix}N&c(N-M)\\ 0&M\end{bmatrix}$ are in $D$, so by Lemma \ref{lem2.20},
\[
f(\begin{bmatrix}N&(B/\epsilon)(N-M)\\ 0&M\end{bmatrix})=\begin{bmatrix}f(N)&(B/\epsilon)(f(N)-f(M))\\ 0&f(M)\end{bmatrix}
\]
In particular, we see that $\norm{(B/\epsilon)(f(N)-f(M))} < B$, or equivalently, $f(N) \in \b{f(M)}{\epsilon}$. This proves \eqref{3.40}

To see that $f$ is holomorphic, fix $M\in D \cap \mn^d$. If $E \in \mn^d$ is selected sufficiently small, then
\[
\begin{bmatrix}M + \lambda E & E \\0&M\end{bmatrix} \in D \cap \m_{2n}^d 
\]
for all sufficiently small $\lambda \in \c$. But
\[
\begin{bmatrix}M + \lambda E & E \\0&M\end{bmatrix}=\begin{bmatrix}M + \lambda E & (1/\lambda)\big((M+\lambda E)-M\big) \\0&M\end{bmatrix}.
\]
Hence, Lemma \ref{2.20} implies that
\be\label{3.50}
f(\begin{bmatrix}M + \lambda E & E \\0&M\end{bmatrix}) =
\begin{bmatrix}f(M + \lambda E) & (1/\lambda)\big(f(M+\lambda E)-f(M)\big) \\0&f(M)\end{bmatrix}.
\ee
As the left hand side of \eqref{3.50} is continuous at $\lambda = 0$, it follows that the 1-2 entry of the right hand side of \eqref{3.50} must converge. As $E$ is arbitrary, this implies that $f$ is holomorphic.
\end{proof}
If $D$ is an nc-domain, we let $H(D)$ denote the collection of locally bounded nc-functions on $D$. In light of Theorem \ref{thm3.10} we refer to the elements of $H(D)$ as \emph{free holomorphic functions}. Likewise, if $\h$ and $\k$ are Hilbert space, we let $H_{\k}(D)$ (resp $H_{\lhk}(D)$) denote the collection of locally bounded $\k$-valued (resp. $\lhk$-valued) nc-functions on $D$.
 \begin{prop}\label{prop3.10}
 Let $D$ be an nc-domain. $H(D)$ equipped with the metric defined in \eqref{3.15} is complete. Furthermore, Montel's Theorem is true, i.e., if $\mathcal{F} \subseteq H(D)$, then $\mathcal{F}$ has compact closure if and only if $\mathcal{F}$ is locally uniformly bounded.
 \end{prop}
 As mentioned above, Montel's Theorem is not true for $H_{\lhk}(D)$ when either $\h$ or $\k$ is infinite dimensional. However the following useful fact in many applications can take its place.
 \begin{prop}\label{prop3.20}
 Let $D$ be an nc-domain and let $\h$ and $\k$ be Hilbert spaces. $H_{\lhk}(D)$ equipped with the metric defined in \eqref{3.15} is complete. Furthermore, if $\langle f^{(k)}\rangle$ is a locally bounded sequence in  $H_{\lhk}(D)$, then there exist an increasing sequence $\langle k_j \rangle$ and $f \in H_{\lhk}(D)$ such that $f^{(k_j)} \wkto f$.
 \end{prop} 

\section{Partial nc-Sets and Functions}
\label{secpncs}

Let $\n$ denote the set of positive integers. We say that $E$ is a \emph{partial nc-set of size n} if
\[
E \subseteq \bigcup_{m=1}^n\m_m^d,
\]
$E \cap \m_n^d \not= \emptyset$, and $M_1 \oplus M_2 \in E$ whenever $M_1 \in E \cap \m_{m_1}^d$, $M_2 \in E \cap \m_{m_2}^d$, and $m_1 +m_2 \le n$. We do not require that partial nc-sets are closed with respect to unitary conjugations.

If $E$ is a partial nc-set, then we say that a function $u:E \to \m^1$ is a \emph{partial nc-function} if
\be\label{4.10}
\forall_{m\in \n}\ M \in E \cap \m_m^d \implies u(M) \in \m_m,\qquad \text{and}
\ee
\be\label{4.20}
\forall_{M_1,M_2 \in E}\ M_1\oplus M_2 \in E \implies u(M_1 \oplus M_2)=u(M_1)\oplus u(M_2).
\ee
In a similar fashion we may define $\k$-valued and $\lhk$-valued partial nc-functions.

If $E$ is a partial nc-set and $\s \subseteq \cup_n \invn$, then we say that a function $u:E \to \m^1$ is \emph{$\s$-invariant} if
\[
\forall_{M \in E}\ \forall_{S \in \s}\ S^{-1}MS \in E \implies u(S^{-1}MS) = S^{-1}u(M)S.
\]
In a similar fashion we may define $\k$-valued and $\lhk$-valued $\s$-invariant functions. Note that the definitions of partial nc-function and $\s$-invariant function are rigged in such a way that $\phi|E$ is an $\s$-invariant partial nc-function whenever $D$ is an nc-domain, $\phi$ is an nc-function on $D$,  $\s \subseteq \cup_n \invn$, and $E \subseteq D$ is a partial nc-set.

We say that $M \in \m_n^d$ is \emph{generic} if there do not exist $M_1,M_2 \in \m^d$ and $S \in \invn $ such that $M = S^{-1}(M_1 \oplus M_2)S$. If $E$ is a partial nc-set, we say that \emph{$E$ is complete} if $M_1 \oplus M_2 \in E$ implies that $M_1,M_2 \in E$. If $M \in E$, we say that \emph{$M$ is E-reducible } if there exist $M_1,M_2 \in E$ such that $M = M_1 \oplus M_2$.
Finally, we shall let $\sin$ denote the $d$-tuples of scalar matrices:
\be
\sin \= \{ ( \alpha^1 \idcn, \dots, \alpha^d \idcn)\ : \ \alpha^r \in \C, 1 \leq r \leq d \} .
\label{eqxd5}
\ee
\begin{defin}\label{def4.10}
Let $E$ be a partial nc-set of size $n$ and $\s \subseteq \cup_n\invn$. For each $m \le n$ let $\calgm$ denote the generic elements of $E \cap \m_m^d$ and let $\calrm$ denote the $E$-reducible elements of $E \cap \m_m^d$. We say the pair $(E,\s)$ is \emph{well organized}, if $E$ is finite and complete,
\be\label{4.50}
\forall_{m\le n}\ E \cap \m_m^d=\calrm \cup \calgm,
\ee
and finally, for each $m\le n$ there exists a set $\calbm \subseteq \calgm$ such that
\be\label{4.60}
\{\calbm\} \cup \set{S^{-1}\calbm S \cap E}{S \in \s \cap \invm} \text{ is a partition of } \calgm \text{ and}
\ee
\be\label{4.70}
\forall_{M \in E\cap\m_m^d}\ \forall_{S \in \s \cap \invm}\ S^{-1}MS \in E \implies M \in \calbm
\cup \Sigma^d_m
\ee
\end{defin}

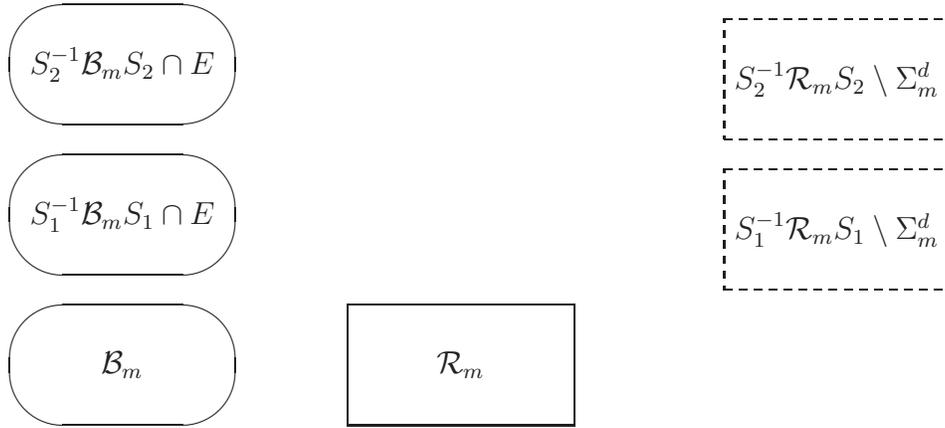
\setlength{\unitlength}{1cm}
\begin{picture}(13,8)
\put(6,1.7){\framebox(3,1.6){$\calrm$}}
\put(11,3.5){\dashbox{.1}(3,1.6){$S_1^{-1} \calrm S_1 \setminus  \Sigma^d_m$}}
\put(11,5.5){\dashbox{.1}(3,1.6){$S_2^{-1} \calrm S_2 \setminus  \Sigma^d_m$}}
\put(3,2.5){\oval(3,1.6)} \put(3,2.5){\makebox(0,0){$\calbm$}}
\put(3,4.5){\oval(3,1.6)} \put(3,4.5){\makebox(0,0){$S_1^{-1} \calbm S_1 \cap E$}}
\put(3,6.5){\oval(3,1.6)} \put(3,6.5){\makebox(0,0){$S_2^{-1} \calbm S_2\cap E$}}
\put(6,1){\makebox(0,0){ A cartoon picture: the solid sets constitute $E_m$. The ovals are $\calgm$.}}
%\label{figd1}
\end{picture}

We note in this definition that necessarily, as the elements of $\calrm$ are $E$-reducible and the elements of $\calgm$ are generic, $\calrm \cap \calgm = \emptyset$. When $m=1$,  $\calrm = \emptyset$, and also, \eqref{4.60} implies that $\calbm = \calgm$ and $\s \cap \invm = \emptyset$. 
Note that for each $m \le n$ and for each $S \in \s \cap \invm$, \eqref{4.60} and \eqref{4.70} imply that
\[
\calbm \= \cup_{S \in \s \cap \invm} (S \calgm S^{-1} \cap \calgm) \cup \{ M \in \calgm \, : \,
\nexists S \in  \s \cap \invm {\rm \ s.t.\ } S^{-1} M S \in E \}
\] (so that $\calbm$ is uniquely determined by $(E,\s)$). When $(E,\s)$ is a well organized pair of size $n$ we set $\calb = \cup_{m \le n}\calbm$ and refer to $\calb$ as the \emph{base of $(E,\s)$}. Similarly, we set $\calg = \cup_{m \le n}\calgm$ and $\calr = \cup_{m \le n}\calrm$.

 If $n \in \n$, we say that $\pi$ is an \emph{ordered partition of n} if there exists a $\sigma \in \n$ such that
\be\label{4.80}
\pi:\{1,\ldots,\sigma\} \to \n\ \ \text{ and } \ \ \sum_{i=1}^\sigma \pi(i) =n.
\ee
We let $\Pi_n$ denote the set of ordered partitions of $n$. If $\pi \in \Pi_n$ and is as in \eqref{4.80}, we set $|\pi| = \sigma$. Finally, we let $[n]$ denote the \emph{trivial} partition, defined by
\[
[n]:\{1\} \to \n\ \ \text{ and } \ \ [n](1) =n.
\]

If $\pi \in \Pi_n$, we let $\m_{n,\pi}^d$ denote the set of $M \in \mn^d$ that have the form
\[
M = \bigoplus_{i=1}^{|\pi|}M_i
\]
where $M_i \in \m_{\pi(i)}^d$ for each $i=1,\ldots,|\pi|$.
\begin{lem}\label{lem4.10}
If $E$ is a partial nc-set, $\s \subseteq \cup_m\invm$, and $(E,\s)$ is well organized of size $n$, then for each $m \le n$, $M \in E \cap \m_m^d$ if and only if there exists a partition $\pi \in \Pi_m$ and matrices $M_1,\ldots,M_{|\pi|}$ such that
\be\label{4.85}
M = \oplus_{i=1}^{|\pi|} M_i \text{ and } M_i \in \mathcal{G}_{\pi(i)} \text{ for } i=1,\ldots,|\pi|.
\ee
 Furthermore, $\pi$ and $M_1,\ldots,M_{|\pi|}$ satisfying \eqref{4.85} are uniquely determined by $M$.
\end{lem}
\begin{proof}
Let $m \le n$ and fix $M \in E \cap \m_m^d$. As $E$ is assumed to be complete, an inductive argument implies that there exist $\pi \in \Pi_m$ and $M_1,\ldots,M_{|\pi|} \in E$ such that
\[
M = \oplus_{i=1}^{|\pi|} M_i,\qquad \forall_i\ M_i \in E \cap \m_{\pi(i)},
\]
where $M_i$ is not $E$-reducible for each $i=1,\ldots,|\pi|$. In particular, \eqref{4.50} implies that $M_i \in \calg_{\pi(i)}$ for each $i$. That the decomposition is unique, follows from the fact that each of the summands, $M_i$ is generic, and hence, irreducible.
\end{proof}
If $(E,\s)$ is a well organized pair of size $n$ , we define $\vspace (E,\s)$ to be the vector space consisting of the $\s$-invariant partial nc-functions on $E$, and define $\grade (\calb)$ to be the vector space of graded matrix valued functions on $\calb$, i.e., the collection of functions $\omega:\calb \to \m^1$ such that
\[
\forall_{m \le n}\ \forall_{M \in \calb \cap \m_m^d}\ \omega(M) \in \m_m.
\]
\begin{prop}\label{prop4.10}
The map $\rho:\vspace (E,\s) \to \grade (\calb)$ defined by $\rho(\phi)=\phi\ |\ \calb$ is a vector space isomorphism.
\end{prop}
\begin{proof}
\eqref{4.10} guarantees that $\rho$ maps into $\grade(\calb)$ and clearly, $\rho$ is linear. To see that $\rho$ is onto, fix $\omega \in \grade(\calb)$. Define $u$ on $\calb$ by setting
\be\label{4.90}
u(x) = \omega(x),\qquad x \in \calb.
\ee
Then use \eqref{4.60} to extend $u$ to $\cup_{m \le n} \calgm$ by the formulas
\be\label{4.100}
u(x) = S^{-1}u(SxS^{-1})S, \qquad m\le n, x,SxS^{-1}  \in \calgm
\ee
where $S$ is the unique element in $\s \cap \invm$ such that $x \in S^{-1}\calbm S$. Finally, we extend $u$ to $\cup_{m \le n} \calrm$ by setting
\be\label{4.110}
u(x) = \oplus_{i=1}^{|\pi|} u(M_i)\qquad m\le n,\ x \in \calrm
\ee
where $x = \oplus_{i=1}^{|\pi|} M_i$ is the unique representation of $x$ given by Lemma \ref{lem4.10}.

To see that $u$ as  just  defined is a a partial nc-function first fix $M_1 \in E \cap \m_{m_1}^d$ and $M_2 \in E \cap \m_{m_2}^d$ where $m_1+m_2 \le n$. By Lemma \ref{4.10}, there exist partitions $\pi_1 \in \Pi_{m_1}$ and $\pi_2 \in \Pi_{m_2}$ such that
\[
M_1 = \oplus_{i=1}^{|\pi_1|} M_i,\qquad  M_i \in \mathcal{G}_{\pi_1(i)} \text{ for } i=1,\ldots,|\pi_1|
\]
and
\[
M_2 = \oplus_{i=1}^{|\pi_2|} N_i,\qquad  N_i \in \mathcal{G}_{\pi_2(i)} \text{ for } i=1,\ldots,|\pi_2|.
\]
If we define $\pi \in  \Pi_{m_1+m_2}$ by
\[
\pi(l) =
\left\{
	\begin{array}{ll}
		\pi_1(l)  & \mbox{if } 1 \le l \le |\pi_1|, \\
		\pi_2(l-|\pi_1|)  & \mbox{if } |\pi_1| +1 \le l \le |\pi_1|+|\pi_2|
	\end{array}
\right.
\]
and let
\[
x_l =
\left\{
	\begin{array}{ll}
		M_l  & \mbox{if } 1 \le l \le |\pi_1|, \\
		N_{l-|\pi_1|}  & \mbox{if } |\pi_1| +1 \le l \le |\pi_1|+|\pi_2|,
	\end{array}
\right.
\]
then $M_1 \oplus M_2 = \oplus_l x_l$ is the unique decomposition of $M_1 \oplus M_2$ given in Lemma \ref{lem4.10}. Hence, using \eqref{4.110},
\begin{align*}
u(M_1 \oplus M_2) &= u(\oplus_l x_l)\\
&= \oplus_l u(x_l)\\
&=\oplus_{l=1}^{|\pi_1|}x_l \oplus \oplus_{l=|\pi_2|+1}^{|\pi_1|+|\pi_2|}x_l\\
&=u(M_1) \oplus u(M_2).
\end{align*}

To see that $u$ is $\s$-invariant, fix $M \in E \cap \m_m^d$ and $S \in \s \cap \invm$ satisfying $S^{-1}MS \in E$. Then \eqref{4.70} guarantees that $M \in \calbm \cup \Sigma^d_m$.
If $M \in \Sigma^d_m$, then so is $u(M)$ by \eqref{4.110}, and both $M$ and $u(M)$
are left invariant by conjugation with $S$.
If $M \in \calbm$, then using \eqref{4.100},
\begin{align*}
u(S^{-1}MS) &= S^{-1} u(S(S^{-1}MS)S ^{-1})S\\
&=S^{-1}u(M)S.
\end{align*}

Summarizing, we have shown that $u$, as defined above, is a partial nc-function that is $\s$-invariant. Hence, $u \in \vspace(E,\s)$. That $\rho(u) = \omega$ follows from \eqref{4.90}. This completes the proof that $\rho$ is onto.

To see that $\rho$ is 1-1, notice that if $v \in \vspace(E,\s)$ and $\rho(v) =\omega$, then as $\rho(v) =\omega$, necessarily \eqref{4.90} holds with $u$ replaced with $v$. As $v$ is $\s$-invariant, \eqref{4.100} also holds with $u$ replaced with $v$. Finally, as $v$ is a partial nc-function, \eqref{4.110} as well holds with $u$ replaced with $v$. These facts imply that $v=u$.
\end{proof}

For the remainder of the section $(E,\s)$ is a well organized pair of size $n$ and we set $\vspace=\vspace(E,\s)$.
We define a $d$-tuple of linear transformations,
\be\label{4.114}
X_\vspace=(X_\vspace^1,\ldots,X_\vspace^d),
\ee
on  $\vspace (E,\s)$ by setting
\be\label{4.115}
(X_\vspace^r u) (x) = x^r u(x),\qquad x \in E.
\ee
Likewise, we define a $d$-tuple of linear transformations $X_\calb=(X_\calb^1,\ldots,X_\calb^d)$ on  $\grade(\calb)$ by setting
\[
(X_\calb^r \omega) (x) = x^r \omega(x),\qquad x \in \calb.
\]
When $\calv_1$ and $\calv_2$ are vector spaces and $L:\calv_1 \to \calv_2$ is a vector space isomorphism, we shall write $\calv_1 \stackrel{L}{\sim} \calv_2$. If in addition, $T_1$ is a $d$-tuple of linear transformations of $\calv_1$, $T_2$ is a $d$-tuple of linear transformations of $\calv_2$, and $T_1 = L^{-1}T_2 L$ we write $T_1 \stackrel{L}{\sim} T_2$. Observe that with these notations, that if $\rho$ is the isomorphism of Proposition \ref{prop4.10}, then
\be\label{4.120}
X_\calv \stackrel{\rho}{\sim} X_\calb.
\ee

For a vector space $V$ we let $V^{(m)} = \oplus_{i=1}^m V$ and if $T$ is a linear transformation of $V$ we set $T^{(m)} = \oplus_{i=1}^m T$. We define $\gamma:\m_m \to {(\c^m)}^{(m)}$ by $\gamma(M) = \oplus_{j = 1}^m M_ j$ where $M_j$ is the $j^{\text{th}}$ column of $M$. If we let $M_x$ denote the operator on $\mn$ defined by $M_x(M) = xM$, then $M_x \stackrel{\gamma}{\sim} x^{(m)}$. It follows that if we define
\[
\beta:\grade(\calb) \to \bigoplus_{m=1}^n \bigoplus_{B \in \calbm} (\c^m)^{(m)}
\]
by the formula
\[
\beta(\omega) = \bigoplus_{m=1}^n \bigoplus_{B \in \calbm} \gamma(u(B)) ,
\]
then $\beta$ is an isomorphism and
\be\label{4.130}
X_\calb \stackrel{\beta}{\sim}\bigoplus_{m=1}^n \bigoplus_{B \in \calbm}B^{(m)}.
\ee

Now assume that $D$ is an nc-domain, $E \subseteq D$, and $\phi$ is an nc-function defined on $D$. We may define a linear transformation $M_\phi$ on $\calv$ by the formula,
 \be\label{4.140}
 (M_\phi u )(x) = \phi(x) u(x),\qquad x \in E.
 \ee
 Noting that
 \[
 \calv\ \stackrel{\beta \circ \rho}{\sim}\ \bigoplus_{m=1}^n \bigoplus_{B \in \calbm} (\c^m)^{(m)},
 \]
 \[
 M_\phi\ \stackrel{\beta \circ \rho}{\sim}\ \bigoplus_{m=1}^n \bigoplus_{B \in \calbm} M_{\phi(B)}^{(m)},
 \]
 and
 \[
\bigoplus_{m=1}^n \bigoplus_{B \in \calbm} M_{\phi(B)}^{(m)} = \phi\Big(\bigoplus_{m=1}^n \bigoplus_{B \in \calbm} M_{B}^{(m)}\Big),
 \]
 we have that
 \[
  M_\phi\ \stackrel{\beta \circ \rho}{\sim}\ =
  \phi\Big(\bigoplus_{m=1}^n \bigoplus_{B \in \calbm} M_{B}^{(m)}\Big).
\]
But
\[
X_\calv\ \stackrel{\beta \circ \rho}{\sim}\ \bigoplus_{m=1}^n \bigoplus_{B \in \calbm}B^{(m)}\ \in\ D.
\]
Hence, $X_\calv \in D^{\sss}$ and $M_\phi = \phi^{\sss}(X_\calv)$.

We summarize what has just been proven in the following proposition.
 \begin{prop}\label{prop4.20}
 Let $D$ be an nc domain and assume that $E \subseteq D$ and $\phi$ is an nc-function on $D$. Also assume that $\s \subset \cup_{m=1}^n \invm$, that $(E,\s)$ is a well organized pair, and let $\calv$ denote the vector space of $\s$-invariant nc-functions on $E$. If the d-tuple of linear transformations on $\calv$, $X_\calv$, is defined by \eqref{4.114} and \eqref{4.115} and the linear transformation on $\calv$, $M_\phi$, is defined by \eqref{4.140}, then
 \[
 M_\phi = \phi^{\sss}(X_\calv).
 \]
 \end{prop} 

%%%%%%%%%%%%%%%%%%%%%%%%%%%%%%
%%%%%%%%%%%%%%%%%%%%%%%%%%%%%%%%

\section{Well Organized Models and Realizations}
\label{secwo}

For $D$ an nc-domain, we let $H^\infty(D)$ denote the bounded nc-functions on $D$. As the elements of  $H^\infty(D)$ are locally bounded, it follows from Theorem \ref{thm3.10} that $H^\infty(D) \subseteq H(D)$. 

Similarly, $H^\i_{\l(\h,\k)}$ denotes the bounded $\l(\h,\k)$-valued nc-functions on $D$, so functions $\Psi$ for which
there is a constant $C$ such that 
\[
 \Psi(x)^* \Psi(x) \ \leq \ C \ \idd
\quad \forall x \in D.\]
We fix for the remainder of this section a matrix $\delta$ of free polynomials. By adding rows or columns of zeroes, if necessary, we can  assume that $\delta$ is actually a square $J$-by-$J$ matrix.
Define $G_\delta$
% and  $K_\delta$ by 
%\begin{eqnarray*}
\[
G_\delta \=  \set{M\in \m^d }{\norm{\delta(M)}<1}, \]
%K_\delta &\=&  \set{M\in \m^d \cap D}{\norm{\delta(M)} \leq 1}.
%\end{eqnarray*}
and assume that $\gdel$ is non-empty.

Let us note that it is possible for $\gdel$ to be empty at lower levels, and non-empty at higher ones.
For example, if $\delta$ is the single polynomial
\[
\delta(x^1,x^2) \= 1 - (x^1 x^2 - x^2 x^1)(x^1 x^2 - x^2 x^1) ,
\]
then $\gdel \cap \M^2_1$ is empty, but $\gdel \cap \M^2_2$ is not.
If this occurs, we just start our constructions at the first $m$ for which $\gdel \cap \M^d_m$
is non-empty.

% We let $B^d$ denote the \emph{matrix polyball}, the nc-domain defined by
%\be
%\label{eqxe1}
%B^d = \set{M \in \m^d}{\norm{M^r} < 1,\ r=1,\ldots,d}.
%\ee

\subsection{$\hv,\rv,\pv$ and $\cv$ }
\label{ssechv}

%In this subsection we shall examine the functions of the form $\phi|E$ where $\phi \in H^\infty(G_\delta)$ and $E$ is a finite partial nc-subset of $G_\delta$. To that end, 
We fix for the remainder of the sub-section a well organized pair $(E,\s)$ of size $n$, with $E \subset \gdel$.
We  fix Hilbert spaces $\h, \k_1$ and $\k_2$, with $\h$ finite dimensional;
we shall let $\M$ denote an arbitrary auxiliary Hilbert space. 
We also fix a pair of functions $\Psi$  in $H^\i_{\l(\h,\k_1)}(\gdel)$
and $\Phi$  in $H^\i_{\l(\h,\k_2)}(\gdel)$
satisfying
\be
\label{eq6w1}
\Psi(x)^* \Psi(x) - \Phi(x)^* \Phi(x)\ \geq \ 0 \qquad \forall \ x \in \gdel .
\ee

We define $\Theta(y,x)$ by
\be
\label{eq6w2}
\Theta(y,x) \=
\Psi(y)^* \Psi(x) - 
 \Phi(y)^* \Phi(x)  .
\ee

We adopt the following notations of the previous section: $\calb$, $\calr$, and $\calg$ for the basic, reducible, and generic elements of $(E,\s)$; and $\calv$ for the vector space of $\s$-invariant partial nc-functions on $E$. We let $\calvh$ (resp. $\calvhm$)  denote the vector space
of $\l(\h)$-valued (resp. $\l(\h,\M)$-valued)  partial nc-functions on $E$.
When $\psi \in \calvm$, $M_\psi$ denotes the operator defined on $\calvhm$ by
\be\label{5.05}
(M_\psi u)(x) = \psi(x)u(x),\qquad x \in E.
\ee
%We let
%\[
%N 
%%= \sum_{m=1}^n |\calbm|m^2
% \=\dim(\calvh).
%\]

We let $\hvh$ denote the set of $\l(\h)$-valued graded functions $h$ on
\be
\label{5.051}
E^{[2]} = \bigcup_{m=1}^n E_m \times E_m
\ee
 that have the form
\[
h(y,x) = \sum_{i=1}^\sigma g_i(y)^*f_i(x),\qquad 1\le m \le n,\ x,y \in E \cap \m_m^d
\]
where $\sigma \in \n$ and $f_i,g_i \in \calvhc$ for $i=1,\ldots,\sigma$. $\hvh$ is a finite dimensional vector space and is a Banach space as well, when equipped with the norm
\[
\norm{h} = \sup_{(y,x) \in E^{[2]}}\norm{h(y,x)}.
\]
We set
\be
\label{eqft1}
\rvh = \set{h\in \hvh}{h(x,y) =h(y,x)^*}
\ee
and define $\pvh$ to consist of the elements $h \in \rvh$ that have the special form
\be\label{5.10}
h(y,x) = \sum_{i=1}^\sigma f_i(y)^*f_i(x),\qquad 1\le m \le n,\ x,y \in E \cap \m_m^d
\ee
where $\sigma \in \n$ and $f_i \in \calvhc$ for $i=1,\ldots,\sigma$. Evidently, $\rvh$ is a real subspace of $\hvh$ and $\pvh$ is a cone\footnote{
By a cone, we mean a convex set closed under multiplication by non-negative real numbers.
This is sometimes called a {\em wedge}.} in $\rv$.
% We let $\kvh = \pvh - \pvh$, \ie all functions in 
%$\hvh$ that have the form
%\[
%h(y,x) = \sum_{i=1}^{\sigma_1} f_i(y)^*f_i(x) -
%\sum_{j=1}^{\sigma_2} g_j(y)^* g_j(y) ,\qquad 1\le m \le n,\ x,y \in E \cap \m_m^d,
%\]
%with $f_i$ and $g_j$ in $\calvhc$.
\begin{lem}\label{lem5.10}
Let $\M$ be a finite dimensional Hilbert space, and 
let $F(y,x)$ be an arbitrary graded $\l(\M)$-valued function on $E^{[2]}$.
Let $N = \dim(\calvhm)$.
Then
if $G$ can be represented in the form
\[
G(y,x) = \sum_{i=1}^\sigma g_i(y)^*F(y,x)g_i(x),\qquad 1\le m \le n,\ x,y \in E \cap \m_m^d,
\]
where $\sigma \in \n$ and $g_i \in \calvhm$ for $i=1,\ldots,\sigma$, then $G$ can be represented in the form
\be\label{5.12}
G(y,x) = \sum_{i=1}^N f_i(y)^*F(y,x)f_i(x),\qquad 1\le m \le n,\ x,y \in E \cap \m_m^d,
\ee
where $f_i \in \calvhm$ for $i=1,\ldots, N$.
\end{lem}
\begin{proof}
Let $\langle e_l(x)\rangle_{l=1}^N$ be a basis of $\calvhm$. For each $i = 1,\ldots,\sigma$, let
\[
g_i(x) = \sum_{l=1}^N c_{il}\ e_l(x).
\]
Form the $\sigma \times N $ matrix $C=[c_{il}]$. As $C^*C$ is an $N \times N$ positive semidefinite matrix, there exists an $N \times N$ matrix $A = [a_{kl}]$ such that $C^*C = A^*A$. This leads to the formula,
\[
\sum_{i=1}^\sigma \overline{c}_{il_1}c_{il_2} = \sum_{k=1}^N \overline{a}_{kl_1}a_{kl_2},
\]
valid for all $l_1,l_2 = 1,\ldots, N$. If $1\le m \le n$ and $ x,y \in E \cap \m_m^d$, then
\begin{align*}
G(y,x) &= \sum_{i=1}^\sigma g_i(y)^*F(y,x)g_i(x)\\
&=\sum_{i=1}^\sigma \big(\sum_{l=1}^N c_{il}\ e_l(y)\big)^*F(y,x)\big(\sum_{l=1}^N c_{il}\ e_l(x)\big)\\
&=\sum_{l_1,l_2=1}^N (\sum_{i=1}^\sigma \overline{c}_{il_1}c_{il_2})e_{l_1}(y)^*F(y,x)e_{l_2}(x)\\
&=\sum_{l_1,l_2=1}^N (\sum_{k=1}^N \overline{a}_{kl_1}a_{kl_2})e_{l_1}(y)^*F(y,x)e_{l_2}(x)\\
&=\sum_{k=1}^N \big(\sum_{l=1}^N a_{kl}\ e_l(y)\big)^*F(y,x)\big(\sum_{l=1}^N a_{kl}\ e_l(x)\big).
\end{align*}
This proves that \eqref{5.12} holds with $f_i = \sum_{l=1}^N a_{il}\ e_l$.
\end{proof}
\begin{lem}\label{lem5.20}
If $h \in \pvh$, $x,y \in E \cap \m_m^d$, and $c,d \in \c^m \otimes \h$, then
\[
|\langle h(y,x)c,d\rangle|^2 \le \langle h(x,x)c,c\rangle \langle h(y,y)d,d\rangle.
\]
\end{lem}
\begin{proof}
Assume that \eqref{5.10} holds.
\begin{align*}
|\langle h(y,x)c,d\rangle|^2 &= \Big|\ip{\sum_{i=1}^\sigma f_i(y)^*f_i(x)c}{d}\Big|^2\\
&=\Big|\sum_{i=1}^\sigma \ip{f_i(x)c}{f_i(y)d}\Big|^2\\
& \le \Big(\ \sum_{i=1}^\sigma\ \norm{f_i(x)c}\ \norm{f_i(y)d}\ \Big)^2\\
&\le \big(\sum_{i=1}^\sigma\ \norm{f_i(x)c}^2\big)\big(\sum_{i=1}^\sigma\ \norm{f_i(y)d}^2\big)\\
&=\big(\ip{\sum_{i=1}^\sigma f_i(x)^*f_i(x)c}{c}\big)\big(\ip{\sum_{i=1}^\sigma f_i(y)^*f_i(y)d}{d}\big)\\
&= \langle h(x,x)c,c\rangle \langle h(y,y)d,d\rangle.
\end{align*}
\end{proof}
If $u \in \nc_{\l(\h, \M \otimes\c^J)}(\gdel)$, then 
we may define $\delta u \in  \nc_{\l(\h, \M \otimes\c^J)}(\gdel)$ by the formula,
\be
\label{eqhb1}
(\delta u)(x) \= (\delta(x) \otimes \idd_{\M}) u(x) \qquad x \in \gdel .
\ee
\begin{defin}\label{def5.05}
We let $\cvh$ and $\ctvh$ be the cones generated in $\rvh$ by the elements in $\rvh$ of the form
\[
u(y)^*[ \idd - \d (y)^*\d (x) ]u(x),\qquad{\rm and} \qquad
u(y)^*[\tau^2\, \idd - \d (y)^*\d (x) ]u(x),
\]
respectively, where  $u \in \calvhoj$, and $\tau$ is such that
\be
\label{eqgx1}
 \rho \ :=\ 
\max \{ \| \d(x) \| \, : \, x \in E \}   \ < \  \tau  \ < \ 1.
\ee
\end{defin}
\begin{prop}\label{prop5.20}
$\cvh$ and $\ctvh$  are closed cones.
\end{prop}
\begin{proof}
By Lemma~\ref{lem5.10}, any element of $\ctvh$ can be represented as a sum 
\be
\label{eqex1}
 \sum_{i=1}^N u_i(y)^*[\tau^2\, \idd - \d (y)^*\d (x) ]u_i(x).
\ee
Suppose a sequence of sums of the form \eqref{eqex1} converges to some element $h(y,x)$  in $\rvh$.
Since $ \rho  < \tau $, we know that each of the
individual  functions 
$u_i$ must eventually satisfy
\[
\| u_i(x) \|^2 \ \leq \ 2\, \frac{1}{\tau^2 - \rho^2} h(x,x). 
\]
So by compactness, a subsequence of the sequence will converge to another sum of the form
\eqref{eqex1}. Letting $\tau =1$ gives the proof for $\cvh$.
\end{proof}
\begin{prop}\label{prop5.30}
$\pvh \subseteq \cvh \subseteq\ctvh$
\end{prop}
\begin{proof}
To prove $\pvh \subseteq \ctvh$,
we must show that for any $f \in \calvhc$, the function $f(y)^* f(x)$ is in $\ctvh$.
Let $g(x) $ be  $ \frac{1}{\tau}f(x)$.
Let $h_\sigma \in \ctvh$ be 
\begin{eqnarray*}
h_\sigma(y,x) &\=&
\sum_{j=0}^\sigma
g(y)^*(\delta(y)^*/\tau)^j [\tau^2 \,  \id{{\c^J}} - \d(y)^* \d(x) ] 
(\delta(x)/\tau)^j g(x) \\
&=& f(y)^* f(x) - g(y)^*(\d(y)^*/\tau)^{\sigma +1}  
(\d(x)/\tau)^{\sigma +1}  g(x).
\end{eqnarray*}
As $\d/\tau$ is a strict contraction on $E$, $h_\sigma(y,x)$
converges to $f(y)^* f(x)$. By Proposition~\ref{prop5.20}, we are done.
Letting $\tau=1$, we get $\pvh \subseteq \cvh$.

To show $ \cvh \subseteq\ctvh$, observe that
\be
\label{eqgx2}
u(y)^*[ \idd - \d (y)^*\d (x) ]u(x) \=
u(y)^*(\tau^2 \, \idd - \d (y)^*\d (x) )u(x)
+ (1-\tau^2) u(y)^* u(x) .
\ee
The first term on the right in \eqref{eqgx2} is in $\ctvh$ by definition, and
the second since $\pvh$ is.
\end{proof}

\begin{lem}\label{lemgx1} For each $\tau$ in $(\rho,1)$, the function $\Theta(y,x) $ is in $\ctvh$.

\end{lem}
\begin{proof}
By Proposition \ref{prop5.20}, $\ctvh$ is a closed cone in $\rvh$. Therefore, by the Hahn-Banach Theorem the lemma will follow if we can show that \be\label{5.20}
L(\Theta(y,x)) \ge 0
\ee
whenever
\be\label{5.30}
L \in \rvh^* \text{ and } L(h) \ge 0 \text{ for all }h \in \ctvh.
\ee
Accordingly assume that \eqref{5.30} holds. Define $L^\sharp \in \hvh^*$ by the formula
\[
L^\sharp (h(y,x)) = L(\frac{h(y,x) + h(x,y)^*}{2})+iL(\frac{h(y,x) - h(x,y)^*}{2i}),
\]
and then define a sesquilinear form on $\calvhc$ by the formula,
\be
\label{eqxe2}
\langle f,g\rangle_L = L^\sharp(g(y)^*f(x)),\qquad f,g \in \calvhc.
\ee

Observe that Proposition \ref{prop5.30} implies that $f(y)^*f(x) \in \ctvh$ whenever $f \in \calvhc$. Hence, \eqref{5.30} implies that
\[
\ip{f}{f}_L \ge 0
\]
for all $f \in \calvhc$, i.e., $\ip{\cdot}{\cdot}_L$ is a pre-inner product on $\calvhc$.
We make this into an inner product by choosing $\vare > 0$, and defining
\[
\langle f,g\rangle_\vare
\=
\langle f,g\rangle_L \ + \
\vare \sum_{x \in E} {\rm tr} (g(x)^* f(x) ) .
\]

% It follows that $\ip{\cdot}{\cdot}_L$ induces an inner product on the quotient, $\calvhc/\mathcal{N}_L$, where
%\be
%\label{eqxe3}
%\mathcal{N}_L = \set{f \in \calvhc}{\ip{f}{f}_L = 0}.
%\ee
We let $\htwole$ denote the Hilbert space $\calvhc$ equipped with the  inner product
$\langle \cdot, \cdots \rangle_\vare$.

Now observe that for each $r=1,\ldots,d$, $x^r \in \calv$, so
% that \eqref{5.05} implies that $M_{x^r}
$X^r_\calv $ 
is a well defined operator on $\calvhc$,
where $X_\calv$ is the linear transformation defined in \eqref{4.114} and \eqref{4.115}.
 \eqref{eqhb1} implies that 
 $\d$ is a well-defined operator from $\calvhc \otimes \c^J$ to $\calvhc  \otimes \c^J$.

If $f \in \calvhoj$, it follows using \eqref{eqgx1} and the fact that
 $f(y)^*(\tau^2 \idd -\d(y)^*\d(x))f(x) \in \ctvh$ that
\begin{eqnarray*}
 \norm{\tau f }_{H_L^{2(J)}}^2 - \norm{\d f}_{H_L^{2(J)}}^2 
%&\=&\ \ip{\tau f}{\tau f}_{H_L^{2(J)}} \  - \ \ip{\d f}{\d f}_{H_L^{2(J)}}\\
&=&\ \tau^2 L^\sharp(f(y)^*f(x)) - L^\sharp((\d(y) f(y))^* \d(x) f(x)) \\
&&\ + \ \vare \sum_{x \in E} {\rm tr}( \tau^2 f(x)^* f(x) -  f(x)^* \d(x)^* \d(x)  f(x))
\\
&=&\ L^\sharp(f(y)^*(\tau^2 \idd -{\d(y)}^*\d(x))f(x))\\
&&\ + \ \vare  \sum_{x \in E} {\rm tr}(f(x)^* ( \tau^2 - \d(x)^* \d(x))  f(x) ) \\
&\ge&\ 0.
\end{eqnarray*}
Hence, the formula,
\[
M_\delta (f) = \delta f ,\qquad f \in \calvhoj,
\]
defines a strict contraction from $\htwolej$ to $\htwolej$.
%for each $r=1,\ldots,d$.  
Let $M_x=(M_{x^1},\ldots,M_{x^d})$ act on $\htwole$, and 
let $\iota:\calvhc \to \htwole$ denote the canonical identity map.
%, $\pi(f) = f+\nl$. Then
%\[
%M_\delta \pi = \pi \delta^\sss
%\]
Then
\[
M_x  \ \iota \= \iota \ X_\calv  .
\]
Hence by Lemma \ref{lem2.30}, for every $\phi$ that is nc on an nc-domain containing $E$,
\[
\phi^{\sss}(M_x ) \iota = \iota \phi^{\sss}( X_\calv ).
\]
Since $\htwole$ is finite dimensional, we have that $M_x$ is unitarily equivalent to some
point $N$ in $\m^d$. Since $\d$ is nc, we have that $\d^\sss(M_x ) = M_\d$ 
%(which is  multiplication by $\delta$ on $\htwole^J$)
 is unitarily equivalent to $\d(N)$.
As $M_\delta$ is a strict contraction, it follows that $N \in \gdel$.

It follows that 
$\Psi^\sss(M_x )$ is unitarily equivalent to $\Psi(N)$, and
$\Phi^\sss(M_x )$ is unitarily equivalent to $\Phi(N)$,
so by \eqref{eq6w1}, multiplication by  $\Phi^\sss(M_x )$ applied to any vector
yields something of smaller norm than multiplication by  $\Psi^\sss(M_x )$.  
Both of these matrices are in $\L(  \htwole \otimes \h, \htwole \otimes \k)$.

Let $\mu = \dim(\h)$, and choose a basis $e_1, \dots, e_\mu$ for $\h$.
Let $f = (f_1 , \dots , f_\mu)^t \in \htwole \otimes \h$
be the vector where $f_j$ is 
%$\pi(0,\dots,0,1,0,\dots,0)$, \ie the 
%representative in  $\calvhc/\mathcal{N}_L$ of the row with
 the constant function $e_j^*$ from $\h$ to $\c$.
% $1$ in the $j^{\rm th}$ slot and
%zero elsewhere.
We get
\beq
\lefteqn{
\la  \Phi^\sss(M_x ) f,   \Phi^\sss(M_x  ) f \ra_{\htwole \otimes \k_2} }\\
\\
 &\=& 
\la \begm
(\Phi_{11}(x) &  \hdots & \Phi_{1 \mu}(x)  ) \\
(\Phi_{2 1}(x) & \hdots & \Phi_{2 \mu }(x) )\\
&\vdots
\endm
\begm
f_1\\
\vdots\\
f_\mu
\endm
 \ , \
\begm
\Phi_{11}(y) &  \hdots & \Phi_{1\mu }(y) ) \\
(\Phi_{2 1}(y) &  \hdots & \Phi_{2 \mu }(y) )\\
&\vdots
\endm
\begm
f_1\\
\vdots\\
f_\mu
\endm
 \ra_{\htwole \otimes \k_2}
\\
 &\=& 
\la \begm
\sum_{j=1}^\mu \Phi_{1j}(x) e_j^* \\
\sum_{j=1}^\mu \Phi_{2j}(x) e_j^* \\
\vdots
%(\Phi_{11}(x), &  \hdots ,& \Phi_{1 \mu}(x)  ) \\
%(\Phi_{2 1}(x), & \hdots, & \Phi_{2 \mu }(x) )\\
%&\vdots
\endm
 \ , \
\begm
\sum_{i=1}^\mu \Phi_{1i}(y) e_i^* \\
\sum_{i=1}^\mu \Phi_{2i}(y) e_i^* \\
\vdots
%(\Phi_{11}(x), &  \hdots , & \Phi_{1\mu }(x) ) \\
%(\Phi_{2 1}(x), &  \hdots , & \Phi_{2 \mu }(x) )\\
%&\vdots
\endm
 \ra_{\htwole \otimes \k_2}
\\
&=&
\sum_k  
\la 
\sum_{j=1}^\mu \Phi_{kj}(x) e_j^*,
\sum_{i=1}^\mu \Phi_{ki}(y) e_i^*
%\Phi_{k1}(x),   \hdots , \Phi_{k\mu }(x)  )\ ,\
%(\Phi_{k1}(x),  \hdots  ,\Phi_{k\mu }(x)  )
 \ra_{\htwole}\\
%&=& 
%\sum_k L^\sharp (
%\sum_{j=1}^\mu \Phi_{kj}(y)^* \Phi_{kj}(x))
%\\
&=&
L^\sharp(\Phi(y)^* \Phi(x)) \ + \
\vare \sum_{x \in E} {\rm tr}(\Phi(x)^* \Phi(x))
\eeq
is smaller than the same expression with $\Psi$ in lieu of $\Phi$.
Letting $\vare \to 0$, we get
\[
L( \Theta(y,x) ) \geq 0 ,
\]
as desired.
\end{proof}

%\end{proof}

\begin{prop}\label{prop5.40}
\[
\Theta(y,x) \= \Psi(y)^*\Psi(x)  -  \Phi(y)^* \Phi(x) \in \cvh.
\]
\end{prop}
\bp
By Lemma~\ref{lemgx1}, we have 
$\Theta \in \ctvh$ for all $\tau$ between $\rho$ and $1$.
By Lemma~\ref{lem5.10}, for each $\tau$ we have 
\be
\label{eqey1}
\Theta(y,x) \=
 \sum_{i=1}^N u^{(\tau)}_i(y)^*[\tau^2\, \idd - \d (y)^*\d (x) ]u^{(\tau)}_i(x).
\ee
As in the proof of Proposition~\ref{prop5.20}, we can use  compactness to extract a sequence 
$\tau_i$ so that the right-hand side of \eqref{eqey1}
converges to an element in $\cvh$.
\ep

%%%%%%%%%%%%%%%%%%%%%%%%%
\subsection{Partial nc-Models}
\label{ssecpn}

%If $u \in \nc_{\c^J}(D)$, then 
%we may define $\delta u $
%by \eqref{eqhb1}.
%%\in  \nc_{\c^I}$ by the formula,
%%\be
%%\label{eqhb1z}
%%(\delta u)(x) = \delta(x)u(x) \qquad x \in D .
%%\ee
%More generally, for each fixed $i_0$ and $j_0$, we can define $E_{i_0j_0}:{\M}^{J} \to \M^I$ by the formula
%\[
%E_{i_0j_0} \oplus_{j=1}^J v_j = \oplus_{i=1}^I w_i
%\]
%where
%\[
%w_i =
%\left\{
%	\begin{array}{ll}
%		v_{j_0}  & \mbox{if } i=i_0, \\
%		0  & \mbox{if } i \not= i_0,
%	\end{array}
%\right.
%\]
%and then consider $\delta \in \nc_{\l({\M^J},{\M^I})}(D)$ defined by
%\[
%\delta (x) = \sum_{i=1}^I \sum_{j=1}^J \delta_{i,j} (x) \otimes E_{ij},\qquad  x \in \m^d.
%\]
%If $u \in \nc_{{\M}^{J}}(D)$, then we may define $\delta u \in  \nc_{{\M}^{I}}(D)$ by the formula,
%\[
%(\delta u)(x) = \delta(x)u(x) \qquad  x \in D.
%\]

\begin{defin}\label{def5.10}
Let $(E,\s)$ be a well organized pair and let
\be
\label{eqgw8}
h(y,x) = \sum_{i=1}^\sigma f_i(y)^* g_i(x),
\ee
 where 
each $f_i$ and $g_i$ are graded $\L(\h,\k_i)$-valued functions on $E$. Assume $E \subset \gdel$.
A \emph{$\delta$-model for $h$} is a  graded $\l(\h,\M\otimes \c^J)$-valued function $u$ on $E$
such that
\be\label{7.10}
h(y,x) \= u(y)^*[ 1-\delta(y)^*\delta(x) ]u(x), \qquad x,y \in E \cap G_\delta
\ee
for all $x,y \in E$. If in addition, $u$ is a partial nc-function, we say the model is \emph{partial nc} and if the model is $\s$-invariant, we say the model is \emph{$\s$-invariant}. If $\M$ is finite dimensional, we say the model is finite dimensional.

If $v$ is a graded $\l(\h,\k)$-valued function on $E$, we say $v$ has a $\d$-model if $[\idd - v(y)^* v(x)]$ does.
\end{defin}
Before continuing we make a few clarifying remarks about the meaning of the formula in \eqref{7.10}. To say that $u$ is an $\s$-invariant partial $\l(\h,\M \otimes \c^J)$-valued nc-function means that
\be\label{5.50}
\forall_{m \le n}\ \forall_{x \in E \cap \m_m^d}\ u(x) \in \l(\c^m \otimes \h,\c^m \otimes \M \otimes \c^J),
\ee
\be\label{5.60}
\forall_{x,y \in E}\ x \oplus y \in E \implies  u(x \oplus y) = u(x) \oplus u(y), \qquad \text{and}
\ee
\be\label{5.70}
\forall_{m\le n}\ \forall_{x \in E \cap \m_m^d}\ \forall_{S \in \s \cap \invm}\ S^{-1} x S \in E \implies u(S^{-1} x S) = (S^{-1} \otimes \idd_{\M \otimes \c^J})u(x) (S\otimes \idh).
\ee
We denote the collection of functions $u$ satisfying these axioms by $\calv_{\l(\h,\M \otimes \c^J)}$.
In the special case when $\M = \ltwo$ or $\ltwo_N$,
%for each $r$ (so that $\k = {\ltwo}^{(d)}$) 
we say the model is \emph{special}. Clearly, as $E$ is finite, if a graded function $v$ has a partial nc-model, then $v$ has a special partial nc-model.
\begin{prop}\label{prop5.50}
Let $(E,\s)$ be a well organized pair, with $E \subseteq G_\delta$.
If $\Theta(y,x)$ is as in \eqref{eq6w2} and is non-negative on $\gdel$ (\ie it satisfies 
\eqref{eq6w1}),   then $\Theta|_{E^{[2]}}$ has an $\s$-invariant finite dimensional partial nc-model.
\end{prop}
\begin{proof}
Let $(E,\s)$ be a well organized pair with $E\subseteq G_\delta$. By Proposition \ref{prop5.40}, $\Theta(y,x) \in \cvh$. Hence, by the definition of $\cvh$ and Lemma~\ref{lem5.10},
\[
\Theta(y,x) \=
u(y)^*[ \idd - \d (y)^*\d (x) ]u(x),
\]
where $u \in \calv_{\l(\h,\ltwo_N^{(J)})}$.
%for $x,y$ in $\gdel$, and therefore by restriction, this holds on $E^{[2]}$.
\end{proof}

\subsection{Partial nc-Realizations}
\begin{defin}\label{def5.20}
Let $(E,\s)$ be a well organized pair of size $n$ and let $\O$ be a graded $\l(\k_1,\k_2)$-valued function defined on $E$.
  A \emph{$\delta$-realization for $\O$} is a pair $(\delta, \mathcal{J})$, where
$\mathcal{J}$ is a finite sequence of operators
\[
\mathcal{J} = \langle J_m \rangle_{m=1}^n=\langle \begin{bmatrix}A_m&B_m\\C_m&D_m\end{bmatrix} \rangle_{m=1}^n
\]
 with $J_m$ acting isometrically from 
$\c^m \otimes \k_1 \oplus (\c^m \otimes \ltwoj)$
to
 $\c^m\otimes \k_2 \oplus (\c^m \otimes \ltwoi)$ for each $m\le n$,
 and such that
\be
\label{eqfw2}
\O(x) = A_m + B_m \delta(x)(\idd -D_m \delta(x))^{-1}C_m
\ee
for each $m\le n$ and $x \in E \cap \m_m^d$. If in addition,
\be
\label{eqfw3}
v(x) \ :=  \ (\idd -D_m \delta(x))^{-1}C_m
\ee
is an $\ltwoi$-valued partial nc-function on $E$ (resp. $(\s \cap \mathcal{I}_m)$-invariant for each $m\le n$), we say that $(\delta,\mathcal{J})$ is \emph{partial nc} (resp. \emph{$\s$-invariant}).
\end{defin}

\begin{thm}\label{thm5.10}
Let $(E,\s)$ be a well organized pair, and let $\Psi \in \calv_{\l(\h,\k_1)}$ and
$\Phi$ be in $ \calv_{\l(\h,\k_2)}$. If there exists a function $\O$
in the closed unit ball of $ \calv_{\l(\k_1,\k_2)}$ that has an $\s$-invariant partial nc $\delta$-realization 
and satisfies $\O \Psi = \Phi$,
then  $[\Psi^*(y) \Psi(x) - \Phi(y)^* \Phi(x)]$ has an $\s$-invariant partial nc-model.
The converse holds if  $\Psi$ is bounded below on $E$.
If $\Psi$ is not bounded below on $E$, then there exists a function $\O$
in the closed unit ball of $ \calv_{\l(\k_1,\k_2)}$ that has a $\delta$-realization 
and satisfies $\O \Psi = \Phi$,
\end{thm}
\begin{proof}  
Suppose $[\Psi^*(y) \Psi(x) - \Phi(y)^* \Phi(x)]$ has an $\s$-invariant partial nc-model, so
there exists  $u \in \calv_{\l(\h,\ltwo_N^{(J)})}$ satisfying
\be
\label{eqfw1}
\Psi^*(y) \Psi(x) - \Phi(y)^* \Phi(x) \=
u(y)^*[ \idd - \d (y)^*\d (x) ]u(x).
\ee
We can rewrite \eqref{eqfw1} to say that
for each $1\leq m \leq n$, the map
\be
\label{eqfw4}
J_m 
\=\begin{bmatrix}A_m&B_m\\C_m&D_m\end{bmatrix}
\, : \,
\begm
\Psi(x) \\
\d(x) u(x)
\endm
\ \mapsto \
\begm
\Phi(x) \\
u(x)
\endm
\ee
is an isometry from the span in $(\k_1 \oplus \ltwo_N^{(J)})\otimes \c^m$
of
\[
\left\{
\ran
\begm
\Psi(x) \\
\d(x) u(x)
\endm \ : \
x \in E \cap \m^d_m
\right\}
\]
to the span in $(\k_2 \oplus \ltwo_N^{(J)})\otimes \c^m$
of
\[
\left\{
\ran
\begm
\Phi(x) \\
 u(x)
\endm \ : \
x \in E \cap \m^d_m
\right\} .
\]
Replacing $\ltwo_N$ by $\ltwo$ if necessary, we can extend $J_m$ to an isometry from all
of $(\k_1 \oplus \ltwo^{(J)})\otimes \c^m$ to $(\k_2 \oplus \ltwo^{(J)})\otimes \c^m$.

Define $\Omega$ and $v$ on $E \cap \m^d_m$ by \eqref{eqfw2}
and \eqref{eqfw3} respectively. Then \eqref{eqfw4} and the fact that $J_m$ is an isometry yield
\begin{eqnarray}
\label{eqfw5}
\O(x) \Psi(x) &\=& \Phi(x) \quad  \forall x \in E \\
\label{eqfw6}
u(x) &\=& v(x) \Psi(x) \quad  \forall x \in E \\
\label{eqfw7}
\idd - \O(y)^* \O(x) &=&
v(y)^* [ \idd - \d(y)^* \d(x) ] v(x) \quad \forall (x,y) \in E^{[2]} .
\end{eqnarray}
Since $u$ is partial nc  on $E$, and $\Psi$ is nc, it would follow from \eqref{eqfw6} that $v$ is also
partial nc on $E$
if $\Psi(x) $ were bounded below.

%If $\Psi(x)$ is not bounded below, and $v$ is not already partial nc, 
%leave $v(x)$ unchanged on each $\calbm$, and extend it to be partial nc on the rest of $E$.

%If $\Psi$ is not bounded below, define new functions $\Psi_\vare$ and
%$\Phi_\vare$ in $H^\i_{\l(\h,\k_1 \oplus \h)}(\gdel)$ and
% $H^\i_{\l(\h,\k_2 \oplus \h)}(\gdel)$ respectively, by
%$\Psi_\vare (x) = \Psi(x) \oplus \vare \idh$ and
%$\Phi_\vare (x) = \Phi(x) \oplus \vare \idh$.
%Repeat the previous construction to get an $\l(\k_1 \oplus \h, \k_2 \oplus \h)$-valued
%graded function $\O_\vare$ on $E$ and a partial nc function $v_\vare$ so that
%\begin{eqnarray}
%\label{eqfw8}
%\O_\vare(x) \Psi_\vare(x) &\=& \Phi_\vare(x) \quad  \forall x \in E \\
%\idd - \O_\vare(y)^* \O_\vare(x) &=&
%v_\vare(y)^* [ \idd - \d(y)^* \d(x) ] v_\vare(x) \quad \forall (x,y) \in E^{[2]} .
%\label{eqfw9}
%\end{eqnarray}
%Writing
%\[
%\O_\vare \=
%\begm \O_{11} & \O_{12}\\
%\O_{21} & \O_{22}
%\endm
%\]
%and using \eqref{eqfw8}:
%\[
%\begm \O_{11} & \O_{12}\\
%\O_{21} & \O_{22}
%\endm
%\
%\begm \Psi & 0 \\
%0 & \vare \idd
%\endm
%\=
%\begm \Phi & 0 \\
%0 & \vare \idd
%\endm,
%\]
%we conclude
%\[
%\O_\vare \=
%\begm \O_{11} & 0\\
%\O_{21} & \idd
%\endm .
%\]
%By \eqref{eqfw9}, 
%\[
%\idd - \O(x)^* \O(x) \=
%\begm \idd - (\O_{11}(x)^* \O_{11}(x) + \O_{21}(x)^* \O_{21}(x)) & -\O_{21}(x)^*\\
%- \O_{21}(x) & 0
%\endm
%\]
%is non-negative, so $\O_{21} = 0$ also.
%Therefore $v_\vare(x) : \c^m \otimes (\k_1 \oplus \h) \to \c^m \otimes (\k_2 \oplus \h)$
%
%
%
%
Conversely, suppose $\O$ existed as in the statement of the theorem.
Then \eqref{eqfw5} and \eqref{eqfw7} would hold, and defining $u(x) := v(x) \Psi(x)$ 
gives \eqref{eqfw1}.
\end{proof}

\begin{remark}
\label{remfw8}
If $\Psi$ is not bounded below, but each $C_m$  and $D_m$ in \eqref{eqfw4} satisfy
$C_m = \idd_{\c^m} \otimes C_1$ and
$D_m = \idd_{\c^m} \otimes D_1$,
 then the converse still holds.
Indeed, follow the above proof through \eqref{eqfw7}. Then define a new $v$ by leaving
$v(x)$ unchanged on $\calbm$, and extending it by Proposition~\ref{prop4.10} to be $\s$-invariant partial nc on $E$. 
Define $\O(x) := A_m + B_m \delta(x) v(x)$.
Since $\Psi$ is nc, \eqref{eqfw6} will still hold, and so will \eqref{eqfw5} and 
\eqref{eqfw7}.
To check \eqref{eqfw3}, we wish to know whether
\[
\idd_{\c^m} \otimes C_1 \= 
(\idd - \idd_{\c^m} \otimes D_1 \delta(x) )^{-1} v(x) .
\]
Both sides are equal on $\calbm$, and both sides are $\s$-invariant partial nc on $E$,
therefore they agree on all of $E$. In Theorem~\ref{thm7.10}, we show that
$C_m$ and $D_m$ can be chosen with this special form.
\end{remark}
%As a corollary to Proposition \ref{prop5.50} and Theorem \ref{thm5.10} we obtain the following result.
%\begin{prop}\label{prop5.60}
%Let $(E,\s)$ be a well organized pair, with $E \subset G_\delta$; and let $\h$ and $\k$ be Hilbert spaces, with $\h$  finite dimensional. If $\Phi \in \ball_{\l(\h,\k)}(H^\infty(G_\delta))$, then $\Phi|E$ has an $\s$-invariant partial nc-realization.
%\end{prop} 

%%%%%%%%%%%%%%%%%%%%%%%%%
%%%%%%%%%%%%%%%%%%%%%%%%%%%%

\section{Full Models and Realizations}
\label{secfm}

%In this section we shall derive a model and a corresponding realization formula for non-negative elements of
%$K(\gdel)$, where $K(\gdel)$ is the set of graded functions in $\gdel^{[2]}$ that have the form
%\be
%\label{eqhw0.5}
%h(y,x) \= 
%f(y)^*f(x) - g(y)^* g(x),
%\qquad \ x,y \in \gdel \cap \m_n^d,
%\ee
%where $f$ and $g$ are graded $\ltwo$-valued functions on $\gdel$.

 Fix again a matrix $\delta$ of nc-polynomials, and assume that $\delta$ is $J$-by-$J$ and that $\gdel$ is non-empty. Let $\h,\k_1,\k_2$, and $\M$ be Hilbert spaces, with $\h, \k_1$ and $\k_2$ finite dimensional.
For the rest of this section, define
\be
\label{eqhw1}
 \Theta(y,x) \=
\Psi(y)^* \Psi(x) - \Phi(y)^* \Phi(x),
%\sum_{i=1}^k \psi_i(y)^* \psi_i(x) - 
%\sum_{j=1}^l \phi_j(y)^* \phi_j(x)  ,
\ee
where
\be
\label{eqhw1.5}
%\forall \ 1 \leq i \leq k,\ 1\leq j \leq l, \
 \Psi\in H^\i_{\l(\h,\k_1)}(\gdel) ,\quad \Phi \ \in  H^\i_{\l(\h,\k_2)}(\gdel),
\ee and 
\be
\label{eqhw2}
 \Theta(x,x) \ \geq \ 0, \qquad \forall x \in \gdel .
\ee
We want to conclude that there exists a function $\Omega$ in the ball of
$H^\i_{\l(\k_1,\k_2)}$
such that
\be
\label{eqhw3}
\Omega(x) \Psi(x) \= \Phi(x), \qquad \forall x \in \gdel .
\ee

\begin{defin}\label{def6.10}
Let $h(y,x)$ be an $\l(\h)$-valued graded function on $\gdel^{[2]}$.
 A \emph{$\delta$-model for $h$}  is a graded $\l(\h,\M \otimes \c^J)$-valued function $u$ on $\gdel$, such that
\be\label{6.10}
h(y,x)  \= u(y)^*[\idd-\delta(y)^*\delta(x)]u(x)
\ee
for all $x,y \in \gdel$. We say the model is \emph{nc (resp. locally bounded, holomorphic)} if $u$ is nc (resp. locally bounded, holomorphic).
\end{defin}
\begin{defin}\label{def6.20}
Let $\Omega$ be a graded $\l(\k_1,\k_2)$ valued function on $\gdel$.
  A \emph{$\delta$-realization for $\O$} is a pair
$(\delta, \mathcal{J} )$, where $\mathcal{J} $ is a 
 sequence of operators
\[
\mathcal{J} = \langle J_m \rangle_{m=1}^\infty=\langle \begin{bmatrix}A_m&B_m\\C_m&D_m\end{bmatrix} \rangle_{m=1}^\infty
\]
 such that $J_m$ acts isometrically from  $\c^{m} \otimes \k_1 \oplus (\c^m \otimes \ltwoj)$
to  $\c^{m}\otimes \k_2 \oplus (\c^m \otimes \ltwoj)$ for each $m$, and such that
\[
\O_m(x) \ :=\ A_m + B_m \delta(x) (\idd -D_m \delta(x))^{-1}C_m.
\]
%satisfies
%\[
%\O_m(x) f(x)
%\=
%g(x)
%\]
 If, in addition,
\[
v(x) = (\idd -D_m \delta(x))^{-1}C_m
\]
is an nc-function on $\gdel$,  we say that $(\delta,\mathcal{J})$ is an \emph{ nc-realization}. 

If, for each $m$, we have
\be
\label{eqgz1}
\begin{bmatrix}A_m&B_m\\C_m&D_m\end{bmatrix}
\=
\begin{bmatrix}\idd_{\c^m} \ot A_1  & \idd_{\c^m} \ot B_1\\
\idd_{\c^m} \ot C_1&\idd_{\c^m} \ot D_1\end{bmatrix},
\ee
we say that $(\delta,\mathcal{J})$ is a \emph{ free realization}. 
\end{defin}

Note that a free realization is automatically an nc-realization.
The following proposition follows by the same lurking isometry argument that proved Theorem \ref{thm5.10}.
\begin{prop}\label{prop6.10}
Let $\O$ be a graded $\l(\k_1,\k_2)$ valued function on $\gdel$. 
Then $\O$  has a $\delta$-realization  if and only if $[\idd - \O(y)^* \O(x)]$ has a 
$\delta$-model; and $\O$ has a $\delta$ nc-realization if and only if $[\idd - \O(y)^* \O(x)]$ has a
$\delta$ nc-model. If $\O$ has a $\delta$-realization, then automatically, the model is both locally bounded and holomorphic.
\end{prop}

\begin{thm}\label{thm6.10}
Let $\Theta$ be as in \eqref{eqhw1} and satisfy 
%\eqref{eqhw1.5} and
\eqref{eqhw2}. 
%Then there exists $\O$ in the closed unit ball of $H^\i_{\l(\k_1,\k_2)}(\gdel)$ such that
%$\O \Psi = \Phi$, and both 
Then  $\Theta$
% and $\idd - \O(y)^* \O(x)$ have
has a  $\delta$ nc-model.
\end{thm}
%A special case of the theorem is when all the functions $\psi_i$ are constant;
%we shall need this to prove the general case.
%
%\begin{thm}\label{thm6.10ad}
%Let $\Phi$ be a graded $\l(\c^k,\c^l)$-valued nc-function on $\gdel$
%with norm bounded by $1$. Then $\Phi$ has a $\delta$ nc-model.
%\end{thm}

The remainder of the section will be devoted to the proof of Theorem \ref{thm6.10}.
This theorem is strengthened in Theorem~\ref{thm7.10}, where it is shown that one can
choose $\Omega$ satsifying \eqref{eqhw3} so that it has a free realization.

When $d =1$, Theorem \ref{thm6.10}  is well-known; see {\em e.g.} 
\cite{ampi} for a treatment in the case of the unit disk. In one variable generalizing to $\gdel$ presents few difficulties.
 
When $d > 1$, in the commutative case, the theorem was first proved by Ambrozie and Timotin in the scalar case in \cite{at03}; it was 
extended to the operator valued case by Ball and Bolotnikov in \cite{babo04}.   See also \cite{amy12b} for an alternative treatment.
In the non-commutative case, Ball, Groenewald and Malakorn \cite{bgm06} proved this theorem for
$\gdel$'s that come from certain bipartite graphs; this includes the most important examples,
the  non-commutative ball
and the non-commutative polydisk.

We shall assume for the rest of this section that 
$d \geq 2$, as the $d=1$ case can be immediately deduced from the $d=2$ case.

%%%%%%%%%%%%%%%%%%%%%%%%%%%%%%%%%%%%
\subsection{Step 1}
In this subsection, for each fixed $n \geq 2$, we shall construct a sequence $(E_\tau,\s_\tau)$ of well-organized partial nc-sets  of size $n$. This will give rise in the next subsection, after taking a cluster point of the sequence, in an holomorphic realization of $\O$ on $\gdel$ that is `nc up to order $n$'.

Fix $n$. Many of the objects in this step of the proof (and steps 2 and 3 as well) will depend on $n$, though our notation will not reflect this fact. For $M\in \mm^d$ we define
\[
\com(M) = \set{A\in \mm}{A\ M^r =M^rA, \text{ for } r=1,\ldots,d}.
\]
\begin{lem}\label{lem6.10}
For each $m=1,\ldots,n$, there exists a sequence, $\langle B_{m,t}\rangle_{t=1}^\infty$ in $\gdel \cap \m_m^d$ such that
\be\label{6.20}
\forall_{t_1,t_2}\  t_1 \not= t_2 \implies B_{m,t_1} \not= B_{m,t_2},
\ee
\be\label{6.30}
\forall_t\  B_{m,t} \text{ is generic},
\ee
\be\label{6.40}
\forall_t\  \com(B_{m,t}) = \c\ \idcm,\qquad \text{and}
\ee
\be\label{6.x50}
\set{B_{m,t}}{t \in \n} \text { is dense in }\gdel \cap \m_m^d.
\ee
\end{lem}
\begin{proof}
 This is easy to verify, because \eqref{6.30} and \eqref{6.40} only fail on sets of lower dimension than
$\m_m^d$.
\end{proof}
% Clear. Can choose first element to be diagonal, with distinct rational  eigenvalues, with respect to some countable dense choice of bases;
% and others to have all non-zero rational entries

Fix sequences $\langle B_{m,t}\rangle_{t=1}^\infty$ satisfying the properties of Lemma \ref{lem6.10}. For each $m=1,\ldots,n$ and $\tau \in \n$, we define
\be\label{6.50}
\calb_{m,\tau} = \set{B_{m,t}}{1 \le t \le \tau}.
\ee
We are going to inductively choose elements $\langle S_{k,m}\rangle_{k=1}^\infty$ in $\invm$,
for $1 \leq m \leq n$.
Once they are chosen,
we define
\[
\s_\tau \= \{ S_{k,m} : \ 1 \leq k \leq \tau,\ 2 \leq m \leq n \} ,
\] and we define
$\calr_{m,\tau}$ to consist of all $R \in \m_m^d$ that have the form,
%\be\label{6.60}
\[
R = \bigoplus_{i=1}^{|\pi|} M_i
\]
where $\pi$ is a nontrivial partition of $m$ and 
\[
M_i \ \in \ \calb_{\pi(i),\tau} \cup \bigcup_{1 \leq k \leq \tau}  \left(   S_{k,\pi(i)}^{-1}  \calb_{\pi(i),\tau}S_{k,\pi(i)}
\cap \gdel \right), \qquad i=1,\ldots, |\pi|.
\]
Note that with this definition, as $\pi$ is required to be nontrivial, $\calr_{1,\tau} = \emptyset$. 
We define $E_{m,\tau} \subseteq \gdel \cap \mm^d$ by
\be\label{6.130}
E_{m,\tau} =\bigcup_{k=1}^\tau \big(\gdel \cap  S_{k,m}^{-1}  \calb_{m,\tau}S_{k,m} \big)\ \cup\ 
\calb_{m,\tau}\ \cup\ \calr_{m,\tau}, \qquad 2 \leq m \leq n.
\ee
We let $E_{1,\tau} = \calb_{1,\tau}$.
Finally, define $E_\tau$  by
\[
E_\tau \= \bigcup_{m=1}^\tau E_{m,\tau}.
\]
\begin{lem}
\label{lem6.x1}
The set $\s_\tau$ can be chosen so that  for each
$ 2 \leq m \leq n$ the set $\{ S_{k,m} : k \in \n \}$ is dense in $\invm$, and

(i) $S^{-1} \calb_{m,\tau} S \subset \calg_m \ \forall \ S \in \s_\tau \cap \invm$.

(ii) $\forall S_{k_1,m}, S_{k_2,m} \in \s_\tau \cap \invm$, the set $S_{k_1,m}^{-1} S_{k_2,m}^{-1}
\calb_{m,\tau}S_{k_2,m} S_{k_1,m}$ is disjoint from $E_{m,\tau}$.

(iii) $ \forall k_1 \neq k_2$ in $\{ 1,2, \dots,  \tau\}$, the set $S_{k_1,m}^{-1} 
\calb_{m,\tau}S_{k_1,m}$ is disjoint from $S_{k_2,m}^{-1} 
\calb_{m,\tau}S_{k_2,m}$ and from $\calb_{m,\tau}$.

(iv) If $R \in R_{m,\tau}$ and for some $1 \leq k \leq \tau$ we have $S_{k,m}^{-1} R S_{k,m} \in E_\tau$, then $R \in \Sigma^d_m$.
\end{lem}
\begin{proof}
This can be done inductively, because each of the conditions holds except on a set in 
$\invm$ of lower dimension than the whole space.
\end{proof}

\begin{lem}\label{lem6.30}
For each $\tau \in \n$, $(E_\tau, \s_\tau)$ is a well organized pair of size $n$.
\end{lem}
\begin{proof}
The necessary conditions follow from Lemma~\ref{lem6.x1}.
\end{proof}

%%%%%%%%%%%%%%%%%%%%%%%%%%%%%%%
\subsection{Step 2}
In this step we shall construct an $\s_\tau$-invariant partial nc-model for $\Theta |E_\tau^{[2]}$ (where $(E_\tau,\s_\tau)$ is the sequence of well organized pairs constructed in step one) that is suitable for forming a cluster point. For each $\tau \in \n$, let $\calv^\tau$ denote the vector space of $\s_\tau$-invariant partial nc-functions on $E_\tau$.

First observe by Proposition \ref{prop5.50} that for each $\tau \in \n$, $\Theta |E_\tau^{[2]}$ has a special finite dimensional model, so there exist
$
u_\tau \in \calv^\tau_{\l(\h,\ltwoj)}
$
such that
\be\label{6.140}
\Theta (y,x) \=  u_\tau(y)^*\big((\idd -{\delta(y)}^*\d(x))\otimes \id{{\ltwo^{(J)}}}\big)u_\tau(x)
\ee
for all $x,y \in E_\tau$.

If $\tau \in \n$, $u \in \calv^\tau_{\l(\h,\ltwoj)}$, and $V$ is a unitary operator acting on $\ltwoj$, we define $V*u \in \calv^\tau_{\l(\h,\ltwoj)}$ by the formula,
\[
V*u (x) = (id_{\c^m} \otimes V) u(x),\qquad 1 \le m \le n,\ x\in E_\tau \cap \m_m^d.
\]
Observe that with this definition, if $V$ is a unitary acting on $\ltwoj$, then \eqref{6.140} holds with $u_\tau$ replaced with $V_\tau * u_\tau$.

Let $\{\xi_1, \dots, \xi_\mu\}$ be a basis for $\h$.
\begin{lem}\label{lem6.40}
Let $\langle M_s\rangle_{s=1}^\sigma$ be a finite sequence in $\m^d$ with $M_s \in \m_{n_s}$ for each $s$.  Let $u$ be a graded $\l(\h,\ltwoj)$ valued function
  on $\set{M_s}{ 1 \leq s \leq \sigma}$. There exists a unitary operator $V$ acting on $\ltwoj$ such that for each $s \le \sigma$,
\be\label{6.150}
\ran ((V*u)(M_s)) \subseteq \c^{n_s} \otimes \ltwoj_{\mu(n_1^2+\ldots+n_s^2  ) }.
\ee
\end{lem}
\begin{proof}
For each $1 \leq r \leq J$, let $u^r$ be the $r^{\rm th}$ component of $u$.
For each $s$,  each $i,j \le n_s$, and each $1 \leq \alpha \leq \mu$, define $\mu n_s^2$ elements $w^r_{s,i,j,\alpha} \in \ltwo$ by
\be\label{6.160}
w^r_{s,i,j,\alpha} = \sum_{l=1}^\infty \langle u^r(M_s)e_j \otimes  \xi_\alpha \, , e_i \otimes \elv\rangle \elv
\ee
In \eqref{6.160}, $\{e_i\}$ is the standard basis for $\c^n$ and $\{\elv\}$ denotes the standard basis for $\ltwo$. For each $s \le \sigma$ define a subspace $\mathcal{W}_s$  of $\ltwo$ by
\[
\mathcal{W}_s^r =\spn  \set{w^r_{s,i,j,\alpha}}{1 \le i,j \le n_s}, \qquad
\mathcal{W}_s = \oplus \mathcal{W}_s^r
\]
and set
\[
\mathcal{X}_s^r   = \mathcal{W}_1^r+ \ldots \mathcal{W}_s^r.
\]
If we set $\nu_s = \max_r \dim \mathcal{X}_s^r$, then there exists a unitary operator acting on $\ltwo$ satisfying $V(\mathcal{X}_1^r) = \ltwo_{\nu_1}$, $V(\mathcal{X}_s^r \ominus \mathcal{X}_{s-1}^r) = \ltwo_{\nu_s} \ominus \ltwo_{\nu_{s-1}}$ for
 $s = 2,\ldots,\sigma$, and $V(\mathcal{X}_\sigma^{r\ \perp}) = \ell_{\nu_\sigma}^{2\ \perp}$. For such a $V$ we have that
\be\label{6.170}
V(\mathcal{X}_s^r) = \ltwo_{\nu_s} \subseteq \ltwo_{\mu(n_1^2+\ldots+n_s^2  )}
\ee
for each $s \le \sigma$.

Now fix $s\le \sigma$ and $j\le n_s$. Using \eqref{6.160} and \eqref{6.170}, we see that\\
\begin{align*}
(V*u)(M_s)(e_j \otimes \xi_\alpha)&= ({\rm id}_{\c^{n_s}} \otimes V) u(M_s)e_j\otimes \xi_\alpha\\ 
&=  ({\rm id}_{\c^{n_s}} \otimes V) \bigoplus_r \sum_{i,l}\langle u^r(M_s)e_j \otimes \xi_\alpha,e_i\otimes \elv\rangle e_i \otimes \elv\\ 
&=\bigoplus_r  \sum_{i}\Big(e_i \otimes V\big(\sum_l\langle u^r(M_s)e_j \otimes \xi_\alpha,e_i\otimes \elv\rangle \elv\big)\Big)\\ 
&=\bigoplus_r  \sum_{i}\Big(e_i \otimes V(w^r_{s,i,j,\alpha})\Big)\\ 
& \in \c^{n_s} \otimes V(\mathcal{W}_s)\\ 
& \subseteq \c^{n_s} \otimes V(\oplus_{r=1}^J \mathcal{X}_s^r)\\ 
& \subseteq \c^{n_s} \otimes \ltwoj_{\mu(n_1^2+\ldots+n_s^2)  }.
\end{align*}
As $e_1,\ldots e_{n_s}$ span $\c^{n_s}$, this proves that \eqref{6.150} holds for each $s\le \sigma$.
\end{proof}
Fix $\tau$  and let $u_\tau$ be as in \eqref{6.140}. We successively enumerate
 the elements of $E_1, E_2\setminus E_1, E_3 \setminus E_2,\ldots, E_\tau \setminus E_{\tau-1}$ and apply Lemma \ref{lem6.40} to obtain a unitary $V_\tau$ and integers $N_t$ (that do not depend on $\tau$) such that for each $t \le \tau$,
\be\label{6.180}
\ran ((V_\tau*u)(x)) \subseteq \c^{m} \otimes \ltwoj_{N_t }\qquad 1 \le m \le N_t,\ x  \in E_t \cap \m_m^d.
\ee
Replacing $u_\tau$ in \eqref{6.140} with $V_\tau * u_\tau$ we thereby obtain the following improvement on \eqref{6.140}.
\begin{lem}\label{lem6.50}
There exists a sequence $\langle N_t\rangle_{t=1}^\infty$ such that for each $\tau \in \n$, there exist
\be\label{6.185}
u_\tau \in \calv^\tau_{\l(\h,\ltwoj)}
\ee
such that
\be\label{6.190}
\Theta(y,x) \=  u_\tau(y)^*\big([\idd -{\d(y)}^*\d(x)]\otimes \id{\ltwo}\big)u_\tau(x)
\ee
for all $x,y \in E_\tau$  and such that for each $t \le \tau$,
\be\label{6.200}
\ran (u_\tau(x)) \subseteq \c^{m} \otimes \ltwoj_{N_t }\qquad 1 \le m \le n,\ x  \in E_t \cap \m_m^d.
\ee
\end{lem}

%%%%%%%%%%%%%%%%%%%%%%%%%%%%%%%%%
\subsection{Step 3}
In this step we shall form a cluster point of the model described in Lemma \ref{lem6.50}. This will result in a model for $\Theta$ on $\gdel$ that is `nc to order n' as described
in Lemma \ref{lem6.60} below.

Fix $\tau$ and let $u_\tau$ be as in Lemma \ref{lem6.50}. 
Note that \eqref{6.185} implies that
\be\label{6.205}
u_\tau(M_1\oplus M_2) = u_\tau(M_1) \oplus u_\tau (M_2)
\ee
whenever $M_1 \in \calb_{m_1,\tau}$, $M_2 \in \calb_{m_2,\tau}$ and $m_1+m_2 \le n$. Also, \eqref{6.185} implies that
\be\label{6.206}
u_\tau(S^{-1}MS) = (S^{-1}\otimes \id{{{\ltwo}^{(J)}}}) u_\tau(M) (S\otimes \idh)
\ee
whenever $M \in \calb_{n,\tau}$, $S \in \mathcal{S}_\tau$, and $S^{-1}MS \in \gdel$.
By Lemma~\ref{lem6.50} and Theorem \ref{thm5.10}, there exist for $m=1,\ldots,n$ isometries
(which depend on $\tau$, though we suppress this in the notation)
\[
\begin{bmatrix}A_m&B_m\\C_m&D_m\end{bmatrix}:\c^{m} \otimes \k_1 \oplus (\c^m \otimes {\ltwoj}) \to \c^{m} \otimes \k_2 \oplus (\c^m \otimes {\ltwoj})
\]
such that for each $m=1,\ldots,n$
\be\label{6.210}
\O_\tau(x) := A_m + B_m \d(x)(\idd -D_m \d(x))^{-1}C_m, \qquad x \in E_\tau \cap \m_m^d
\ee
satisfies
\be\label{6.211}
\O_\tau(x) \Psi (x) \= \Phi(x) ,
\ee
and
\be\label{6.220}
v_\tau(x)  \ :=\  (\idd -D_{m} \d(x))^{-1}C_m,
\ee
satisfies
\be\label{6.2201}
v_\tau(x)  \Psi(x) \= u_\tau(x)
 \qquad x \in E_\tau \cap \m_m^d.
\ee

For $m > n$, choose 
\[
\begin{bmatrix}A_m&B_m\\C_m&D_m\end{bmatrix}:\c^{m} \otimes \k_1 \oplus (\c^m \otimes {\ltwoj}) \to \c^m \otimes \k_2 \oplus (\c^{m} \otimes {\ltwoj})
\]
to be an arbitrary isometry.

Define an $\l(\k_1,\k_2)$-valued  graded functions $\O_\tau$,
 an $\l(\k_1,\ltwoj)$-valued graded function  $V_\tau$, 
and  an $\l(\h,\ltwoj)$-valued graded function  $U_\tau$, 
on $\gdel$ by the formulas
\be\label{6.230}
\O_\tau(x) = A_m + B_m \d(x)(\idd -D_m \d(x))^{-1}C_m, \qquad m\in \n,\ x \in \gdel \cap \m_m^d
\ee
\be\label{6.240}
V_\tau(x) = (\idd -D_{m} \d(x))^{-1}C_m, \qquad m\in \n,\ x \in \gdel \cap \m_m^d
\ee
\be\label{6.2402}
U_\tau(x) = V_\tau(x) \Psi(x), \qquad m\in \n,\ x \in \gdel \cap \m_m^d
\ee
Note that with these definitions that
\be\label{6.245}
\idd-\O_\tau(y)^*\O_\tau(x) = V_\tau(y)^*(\idd -\d(y)^*\d(x))V_\tau(x)
\ee
whenever $m \in \n$ and $x \in \gdel \cap \m_m^d$.

It follows easily from \eqref{6.230} and \eqref{6.240} that $\langle \O_\tau \rangle_{\tau=1}^\infty$ and $\langle V_\tau \rangle_{\tau=1}^\infty$ are uniformly locally bounded sequences of holomorphic functions on $\gdel$. Hence, by Proposition \ref{prop3.30} there exist a subsequence $\tau_j$ and holomorphic functions $\O$ and $U$ such that
\be\label{6.250}
\O_{\tau_j} \to \O
\ee
and
\be\label{6.260}
V_{\tau_j} \stackrel{\text{wk}}{\to} V.
\ee
Let $U = V \Psi$.
Now notice that \eqref{6.210} and \eqref{6.230} imply that \be\label{6.265}
\O_\tau|E_\tau = \O|E_\tau
\ee
for each $\tau$. Hence, as both $\O \Psi$ and $\Phi$ are holomorphic, \eqref{6.x50} and \eqref{6.250} imply that
\be\label{6.270}
%\O_{\tau_j}(x)  \to \O(x)
\O(x) \Psi(x) \= \Phi(x)
\ee
for each $m\le n$ and  $x \in \gdel \cap \m_m^d$. Also notice that \eqref{6.220}, \eqref{6.2201} and \eqref{6.240} imply that
\be\label{6.275}
U_\tau|E_\tau \= u_\tau |E_\tau \= U|E_\tau
\ee
for each $\tau$. Hence, it follows from \eqref{6.200} that if $m \le n$, $t \le \tau_j$ and $x \in E_t \cap \m_m^d$, then
\[
\ran(U_{\tau_j}(x)) \subseteq \c^m \otimes {\ltwoj_{{N_t}}}
\]
Therefore, by \eqref{6.2402} and \eqref{6.260},
\be\label{6.280}
U_{\tau_j}(x) \to U(x)
\ee
whenever $t \in \n$, $m\le n$ and  $x \in E_t \cap \m_m^d$. Combining \eqref{6.245}, \eqref{6.270}, and \eqref{6.280} gives that
\be\label{6.290}
\Psi(y)^*\Psi(x) -\Phi(y)^*\Phi(x) \=
 U(y)^*[\idd -\d(y)^*\d(x)]U(x)
\ee
whenever $x,y \in \cup_{\tau=1}^\infty E_\tau$. As both the right and left hand sides of \eqref{6.290} are holomorphic in $x$ and coholomorphic in $y$, it follows that
\be\label{6.300}
\forall_{m \le n}\ \forall_{x \in \gdel \cap \m_m^d}\ 
\Psi(y)^*\Psi(x) -\Phi(y)^*\Phi(x) \=
 U(y)^*[\idd -\d(y)^*\d(x)]U(x)
\ee
Two additional properties of $U$, as constructed above, are described in the following definition.
\begin{defin}\label{def6.30}
Let $D$ be an nc-domain. We say that \emph{$U$ is an $\l(\h,{\ltwo}^{(J)})$-valued nc-function to order $n$ on $D$} if $U$ is a graded $\l(\h,{\ltwo}^{(J)})$-valued   function defined on $D \cap \cup_{m \le n} \m_m^d$, $U$ is holomorphic,
\be\label{6.310}
x_1 \in D \cap \m_{m_1}^d,\ x_2 \in D \cap \m_{m_2}^d,\ m_1+m_2 \le n \implies U(x_1 \oplus x_2)=U(x_1) \oplus U(x_2),
\ee
and
\be\label{6.320}
m \le n,\ x \in D \cap \m_m^d, S \in \invm,\ S^{-1}x S \in D \cap \m_m^d \implies U(S^{-1}x S)=(S^{-1} \otimes {\id{{\ltwoj}}})\ U(x) (S \otimes \idh).
\ee
\end{defin}
The definition is made for a general nc-domain $D$. We wish to show that \eqref{6.310} and  \eqref{6.320} hold when $D = \gdel$ and $U$ is as constructed above.

To prove \eqref{6.310} assume that $M_1 \in \calb_{m_1,t}$ and $M_2 \in \calb_{m_2,t}$ where $m_1+m_2 \le n$. Then
\begin{align*}
\ U(M_1 \oplus M_2)&\\
\eqref{6.280}\qquad=&\lim_{j\to \infty} U_{\tau_j}(M_1 \oplus M_2)\\
\eqref{6.275}\qquad=&\lim_{j\to \infty} u_{\tau_j}(M_1 \oplus M_2)\\
\eqref{6.205}\qquad =&\lim_{j\to \infty} u_{\tau_j}(M_1) \oplus u_{\tau_j}(M_2)\\
\eqref{6.275}\qquad =&\lim_{j\to \infty} U_{\tau_j}(M_1) \oplus U_{\tau_j}(M_2)\\
\eqref{6.280}\qquad =&U(M_1)\oplus U(M_2).
\end{align*}
Hence, as $U$ is holomorphic, \eqref{6.x50} implies that \eqref{6.310} holds.

To prove \eqref{6.320} assume that $M \in \calb_{m,t}$,  $S  \in \s_t $ and  $S^{-1}MS \in \gdel$ (so that by \eqref{6.130}, $S^{-1}MS \in E_{m,t}$). Then
\begin{align*}
\ U(S^{-1}MS)&\\
\eqref{6.280}\qquad=&\lim_{j\to \infty} U_{\tau_j}(S^{-1}MS)\\
\eqref{6.275}\qquad=&\lim_{j\to \infty} u_{\tau_j}(S^{-1} MS)\\
\eqref{6.206}\qquad =&\lim_{j\to \infty}( S^{-1} \otimes {\id{{\ltwoj}}})\ u_{\tau_j}(M) (S\otimes \idh)\\
\eqref{6.275}\qquad =&\lim_{j\to \infty} (S^{-1} \otimes {\id{{{\ltwoj}}}})\ U_{\tau_j}(M) (S\otimes \idh)\\
\eqref{6.280}\qquad =&S^{-1} \otimes {\id{{\ltwoj}}}\ U(M)S\otimes \idh.
\end{align*}
%Hence, as $U$ is holomorphic, \eqref{6.x50} and Lemma~\ref{lem6.x1} imply that \eqref{6.320} holds in the special case when $m=n$. 
% If $m<n$, $\ x \in \gdel \cap \m_m^d$, $S \in \invm$, and $\ S^{-1}x S \in \gdel \cap \m_m^d$, choose $N \in \gdel \cap \m_{n-m}$ and let $M_1=M \oplus N$ and $S_1 = S \oplus \id{{\c^{n-m}}}$. As \eqref{6.320} holds when $m=n$, it follows that
%\[
%U(S_1^{-1}M_1 S_1)=S_1^{-1} \otimes {\id{{\ltwoj}}}\ U(M_1)S_1.
%\]
%But \eqref{6.310} implies that
%\[
%U(S_1^{-1}M_1 S_1) = U(S^{-1}M S) \oplus U(N)
%\]
%and
%\[
%S_1^{-1} \otimes {\id{{\ltwoj}}}\ U(M_1)S_1=\Big( S^{-1} \otimes {\id{{\ltwoj}}}\ U(M)S\Big) \oplus \Big(\id{{\c^{n-m}}} \otimes {\id{{\ltwoj}}}\ U(N)\Big).
%\]
%Hence,
%\[
%U(S^{-1}M S) = S^{-1} \otimes {\id{{\ltwoj}}}\ U(M)S\Big.
%\]

%\blue
%{\em
%The original did this just for $n$, then used a direct sum argument
%to get it for $m < n$; but I don't see the need.}
%\black

The following lemma summarizes what has been proved. The lemma is expressed in a notation that reflects the dependence of $U$ on $n$.
\begin{lem}\label{lem6.60}
Suppose $\Psi$ is an $\l(\h,\k_1)$ valued nc-function on $\gdel$,
  $\Phi$ is an $\l(\h,\k_2)$-valued nc-function on $\gdel$, and suppose that 
$\Theta(x,x) = \Psi(x)^* \Psi(x) - \Phi(x)^* \Phi(x) \geq 0$ on $\gdel$.
For each $n \in \n$  there exists $U_n$, such that $U_n$ is an ${\ltwoj}$-valued nc-function to order $n$ on $\gdel$, and such that
% and $\O$, an holomorphic function satisfying
%$\O(x) \Psi(x) = \Phi(x)$, 
\be\label{6.330}
\Psi(y)^*\Psi(x) -\Phi(y)^*\Phi(x) \= U_n(y)^*[\idd -\d(y)^*\d(x)]U_n(x)
\ee
\end{lem}

%%%%%%%%%%%%%%%%%%%%%%%%%%%%%%%%%%%
\subsection{Step 4}
In this step we complete the proof that $\Theta$ has a $\delta$-model by taking a cluster point of the `order $n$' models described in Lemma \ref{lem6.60}.

Let $\langle U_n \rangle_{n=1}^\infty$ be a sequence with $U_n$ as in Lemma \ref{lem6.60} for each $n \in \n$. For each $n \in \n$, choose a dense sequence $\langle M_{n,\tau} \rangle_{\tau = 1}^\infty$ in $\gdel \cap \mn^d$ and  a dense sequence $\langle S_{n,\tau} \rangle$ in $\invn$. 
As in the proof of Lemma \ref{lem6.50} we may employ Lemma \ref{lem6.40} to obtain a sequence of unitaries  $\langle V_n \rangle_{n=1}^\infty$ acting on ${\ltwo}^{(J)}$ such that if we define $W_n = V_n*U_n$, then $W_n$ satisfies the conditions of Lemma \ref{lem6.60} and in addition satisfies
\be\label{6.340}
\forall_{n\in \n}\ \exists_{N}\ \forall_{m \le n}\ \forall_{s,t \le m}\  \ran{W_m (M_{s,t})} \subseteq \c^s \otimes \ltwo_{N}^{(J)}.
 \ee
Hence, if we use Proposition \ref{prop3.30} to obtain an $\l(\h,{\ltwo}^{(J)})$-valued holomorphic graded function $W$ on $\gdel$ and a subsequence $\langle n_j\rangle$ such that
\[
W_{n_j} \wkto W,
\]
then
\be\label{6.350}
\forall_{n,\tau \in \n}\ W_{n_j}(M_{n,\tau}) \to W(M_{n,\tau}).
\ee
(Note that \eqref{6.350} is in finite dimensions, so weak convergence gives norm convergence).
To see that $W$ gives rise to an nc-model for $\Theta$ we need to prove the following three assertions:
\be\label{6.360}
\Psi(y)^*\Psi(x) -\Phi(y)^*\Phi(x) = W(y)^*[\idd -\d(y)^*\d(x)]W(x)
\ee
whenever $n \in \n$ and $x,y \in \gdel \cap \mn^d$,
\be\label{6.370}
W(x_1 \oplus x_2) = W(x_1) \oplus W(x_2)
\ee
whenever $n_1,n_2 \in \n$, $x_1\in \gdel \cap \m_{n_1}^d$, and $x_2\in \gdel \cap \m_{n_2}^d$, and
\be\label{6.380}
W(S^{-1}x S) = (S^{-1} \otimes {\id{{\ltwoj}}})\ W(x)S
\ee
whenever $n \in \n$, $x \in \gdel \cap \mn^d$, $S \in \invn$, and $S^{-1}xS \in \gdel$.

To see that \eqref{6.360} holds observe that \eqref{6.330} and \eqref{6.350} imply that \eqref{6.360} holds for each $n$ whenever $x,y \in \set{M_{n,\tau}}{\tau \in \n}$. Hence, as  $x,y \in \set{M_{n,\tau}}{\tau \in \n}$ is dense in $\gdel$ and both sides of \eqref{6.360} are holomorphic in $x$ and coholomorphic in $y$, in fact, \eqref{6.360} holds for all $x,y \in \gdel \cap \mn^d$.

\eqref{6.370} follows by  noting that \eqref{6.310} and \eqref{6.350} imply that \eqref{6.370} holds whenever   $x_1 \in \set{M_{n_1,\tau}}{\tau \in \n}$ and $x_2 \in \set{M_{n_2,\tau}}{\tau \in \n}$. Hence, by density and continuity, \eqref{6.370}  holds for all  $x_1\in \gdel \cap \m_{n_1}^d$ and $x_2\in \gdel \cap \m_{n_2}^d$. Likewise, \eqref{6.380} follows from \eqref{6.320} and  \eqref{6.350}.

This proves Theorem~\ref{thm6.10}. \ep

%\subsection{Step 5}
%
%Combining the previous results, we have found a $\d$ nc-model for $\Theta$, so
%\be
%\label{eq63f1}
%\Psi(y)^*\Psi(x) -\Phi(y)^*\Phi(x) \=
%U(y)^* (\idd - \d(y)^* \d(x) ) U(x) ,
%\ee
%where $U$ is an nc-function.
%Also, we know $\O \Psi = \Phi$.
%
%To here
%
%So conjugate both sides of \eqref{eq63f1} by $\Psi$, and we get
%\[
%\O(y,x) \= \Psi(y)^* \Psi(x) - \Phi(y)^* \Phi(x) \=
%\Psi(y)^* U(y)^* (\idd - \d(y)^* \d(x) ) U(x)  \Psi (x) .
%\]
%This gives a $\d$ nc-model for $\Psi$, with $u = U \Psi$.

\section{$\d$ nc-models and nc-realizations}
\label{ssecdnc}

\begin{thm}\label{thm7.10}Let $\h,\k_1,\k_2$ be finite dimensional Hilbert spaces.
 Let $\delta$ be an $I\times J$ matrix with entries in $\pd$,  %let  $\$ be a graded $\L(\c^k,\c^l)$-valued function on $G_\delta$. 
let $\Psi$ be a graded $\l(\h,\k_1)$-valued function on $\gdel$,
and let  $\Phi$ be graded $\l(\h,\k_2)$-valued  function on $G_\delta$. 
Let $\Theta(y,x) = \Psi(y)^* \Psi(x) - \Phi(y)^* \Phi(x)$.
The following are equivalent.\\
\qquad (1) $\Theta(x,x) \geq 0$ on $\gdel$.\\
%$\Phi \in \ball(H^\infty(G_\delta))$\\
\qquad (2) $\Theta$ has a $\delta$ nc-model.\\
\qquad (3) There exists an nc $\l(\k_1,\k_2)$-valued function $\O$  satsifying
$\O \Psi = \Phi$ and such that $\O$ has a  free $\delta$-realization.\\
\end{thm}
\begin{proof}  (1) implies (2) by Theorem~\ref{thm6.10}. (3) implies (1) because by Proposition~\ref{prop6.10},
we have 
\[
\idd - \O(y)^* \O(x) \= v(y)^* [ 1  - \delta(y)^* \delta(x) ] v(x) .
\]
Multiply by $\Psi(y)^* $ on the left and $\Psi(x)$ on the right, then restrict to the diagonal, to get
$\Theta(x,x) \geq 0$.

Assume that (2) holds, i.e.,
\be
\label{eqhx1}
\Psi(y)^* \Psi(x) - \Phi(y)^* \Phi(x) \= u(y)^* [ 1  - \delta(y)^* \delta(x) ] u(x) 
\ee
 holds, where $u$ is an  $\l(\h,{\ltwo}^{(J)})$-valued nc-function on $G_\delta$. Observe that if $n \in \n$, $S \in \invn$, and we replace $x$ with $S^{-1}xS$ in \eqref{eqhx1}, then
\[
\Psi(y)^*(S^{-1} \otimes \idd_{\k_1})\Psi(x)
-\Phi(y)^*(S^{-1} \otimes \idd_{\k_2}) \Phi(x) =
 u(y)^*\Big(S^{-1} \otimes \id{{\ltwoj}}-\delta(y)^*(S^{-1} \otimes \id{{\ltwoi}})\delta(x)\Big)u(x).
\]
Hence, as $\invn$ is dense in $\mn$, we obtain that in fact,
\be\label{7.20}
\Psi(y)^*(C \otimes \idd_{\k_1})\Psi(x) -
\Phi(y)^*(C \otimes \idd_{\k_2})\Phi(x)
\= u(y)^*\Big(C \otimes \id{{\ltwoj}}-\delta(y)^*(C \otimes \id{{\ltwoi}})\delta(x)\Big)u(x),
\ee
for all $C \in \mn$. For $k=1,\ldots,n$, define $\pi_k:\c^n \to \c$ by the formula
\[
\pi_k(v) = v_k,\qquad v = (v_1,\ldots,v_n) \in \c^n.
\]
Letting $C=\pi_l^*\pi_k$ in \eqref{7.20} and applying to $v\ot \eta$ and $w\ot \xi$,
with $v,w$ in $\c^n$ and $\eta, \xi$ in $\h$, leads to
\begin{align}
&\ip{\pi_k \ot \idd_{\k_1}\Psi (x)v \ot \eta}{\pi_l \ot \idd_{\k_1} \Psi(y) w \ot \xi} - \ip{\pi_k \ot \idd_{\k_2} \Phi(x)v \ot \eta}{\pi_l\ot \idd_{\k_2} \Phi(y) w \ot \xi}\notag\\
=\ &\ \ip{(\pi_k \otimes \id{{\ltwoj}})u(x)v \ot \eta}{(\pi_l \otimes \id{{\ltwoj}})u(y)w \ot \xi} \label{7.30}\\
&- \ip{(\pi_k \otimes \id{{\ltwoi}})\delta(x)u(x)v \ot \eta}{(\pi_l \otimes \id{{\ltwoi}})\delta(y)u(y)w \ot \xi}.\notag
\end{align}
For each $k=1,\ldots,n$, each $v \in \c^n$, each $\eta \in \h$, and each $x \in G_\delta \cap \mn^d$ define a vector $p_{k,v,\eta,x} \in \k_1 \oplus \ltwoi$ by
\[
p_{k,v,\eta,x} = (\pi_k \ot \idd_{\k_1} ) \Psi(x) (v \ot \eta) \oplus (\pi_k \otimes \id{{\ltwoi}})\delta(x)u(x)(v \ot \eta).
\]
Also, define  $q_{k,v,\eta,x} \in \k_2 \oplus \ltwoj$
\[
q_{k,v,\eta,x} = (\pi_k \ot \idd_{\k_2}) \Phi(x) (v \ot \eta) \oplus (\pi_k \otimes \id{{\ltwoj}})u(x)(v \ot \eta).
\]
In terms of the vectors, $p_{k,v,\eta,x}$ and $q_{k,v,\eta,x}$, \eqref{7.30} can be rewritten in the form,
\be\label{7.40}
\ip{p_{k,v,\eta,x}}{p_{l,w,\xi,y}} = \ip{q_{k,v,\eta,x}}{q_{l,w,\xi,y}}.
\ee
Hence, if we let
\[
\mathcal{P}_n = \spn\set{p_{k,v,\eta,x}}{k \le n,v\in \c^n, \eta \in \h, x\in G_\delta \cap \mn^d}
\]
and let
\[
\mathcal{Q}_n = \spn\set{q_{k,v,\eta,x}}{k \le n,v\in \c^n, \eta \in \h, x\in G_\delta \cap \mn^d},
\]
then there exists an isometry $L_n:\mathcal{P}_n \to \mathcal{Q}_n$ satisfying
\be\label{7.50}
L_n \, p_{k,v,\eta,x} \=q_{k,v,\eta,x}
\ee
for all $k,v$, and $x$.

Now let $n \le m$. Fix $k \le n$, $v \in \c^n$, $\eta \in \h$, and $x \in G_\delta \cap \mn^d$. Choose $v_0 \in \c^{m-n}$  and $x_0 \in G_\delta \cap \m_{m-n}^d$ and then define $v_1=v \oplus v_0$ and $x_1=x \oplus x_0$. We have that
\begin{align*}
p_{k,v_1,\eta, x_1}&=( \pi_k \ot \idd_{\k_1})  \Psi(x) (v_1  \ot \eta) \oplus (\pi_k \otimes \id{{\ltwoi}})\delta(x_1)u(x_1) (v_1 \ot \eta) \\
&=(\pi_k\ot \idd_{\k_1}) ( \Psi(x) \oplus \Psi(x_0) ) (v \ot \eta\oplus v_0\ot \eta ) \\
&\qquad \oplus (\pi_k \otimes \id{{\ltwoi}})\delta(x\oplus x_0)u(x \oplus x_0)(v \ot \eta\oplus v_0 \ot \eta)\\
&=(\pi_k\ot \idd_{\k_1}) ( \Psi(x) v \ot \eta \oplus \Psi(x_0)  v_0\ot \eta ) \\
&\qquad \oplus (\pi_k \otimes \id{{\ltwoi}})(\delta(x)u(x )v \ot \eta \oplus  \delta(x_0)u(x_0 ) v_0 \ot \eta)\\
&=(\pi_k\ot \idd_{\k_1}) ( \Psi(x) v \ot \eta ) 
 \oplus (\pi_k \otimes \id{{\ltwoi}})(\delta(x)u(x )v \ot \eta )\\
&=p_{k,v,\eta,x}
\end{align*}
This shows that if $n \le m$, $k \le n$, $v \in \c^n$, $\eta \in \h$ and $x \in G_\delta \cap \mn^d$, then $p_{k,v,\eta,x} =p_{k,v_1, \eta,x_1} \in \mathcal{P}_m$. Therefore, 
\be\label{7.60}
\mathcal{P}_n \subseteq \mathcal{P}_m
\ee
whenever $n \le m$. In like fashion, if $k \le n$, $v \in \c^n$, $\eta \in \h$ and $x \in G_\delta \cap \mn^d$, then $q_{k,v,\eta,x} =q_{k,v_1,\eta,x_1}$ so that
\be\label{7.70}
\mathcal{Q}_n \subseteq \mathcal{Q}_m.
\ee
Finally, observe that when $k \le n$, $v \in \c^n$ and $x \in G_\delta \cap \mn^d$ and $v_1$ and $x_1$ are as defined above,
\begin{align*}
L_n\ p_{k,v,\eta,x} &= q_{k,v,\eta,x}\\
&=q_{k,v_1,\eta,x_1}\\
&=L_m\ p_{k,v_1,\eta,x_1}\\
&=L_m\ p_{k,v,\eta,x}.
\end{align*}
Therefore, when $n \le m$,
\be\label{7.80}
L_n = L_m | \mathcal{P}_n.
\ee
Let $\mathcal{P} = (\cup_{n=1}^\infty\mathcal{P}_n)^- \subseteq \k_1 \oplus \ltwoi$ and $\mathcal{Q} = (\cup_{n=1}^\infty\mathcal{Q}_n)^- \subseteq \k_2 \oplus \ltwoj$.  \eqref{7.60}, \eqref{7.70}, and \eqref{7.80} together imply that there exists an isometry $L:\mathcal{P} \to \mathcal{Q}$ such that
\be\label{7.90}
Lp_{k,v,\eta,x} =q_{k,v,\eta,x}
\ee
whenever $n\in \n$, $k \le n$, $v \in \c^n$, $\eta \in \h$, and $x \in G_\delta \cap \mn^d$. By replacing $u$ in \eqref{eqhx1} with $(\id{{\c^n}} \otimes \tau^{(J)}) u$ where $\tau:\ltwo \to \ltwo$ is an isometry with $\ran(\tau)$ having infinite codimension in
 $\ltwo$ we may ensure that $\mathcal{P}$ has infinite codimension in  $\k_1 \oplus \ltwoi$ and $\mathcal{Q}$ has infinite codimension in  $\k_2 \oplus \ltwoj$. Hence, there exists an isometry (or even a Hilbert space isomorphism) $J_1:\k_1 \oplus \ltwoi \to \k_2 \oplus \ltwoj$ such that
\be\label{7.902}
J_1p_{k,v,\eta,x} =q_{k,v,\eta,x}
\ee
whenever $n\in \n$, $k \le n$, $v \in \c^n$ and $x \in G_\delta \cap \mn^d$.

There remains to show that $J_n=\id{{\c^n}}  \otimes J_1$ defines an nc-realization of $\Omega$.
First, let us show that
\be
\label{eqhy2}
(\id{{\c^n}}  \otimes J_1 )
\begm
\Psi(x) \\
\delta(x) u(x) 
\endm
\=
\begm
\Phi(x) \\
u(x)
\endm .
\ee

 Fix $n \in \n$, $v \in \c^n$ $\eta \in \h$, and $x \in G_\delta \cap \mn^d$.
\begin{align*}
&(\id{{\c^n}}\otimes J_1)\big( \Psi(x)  v \otimes \eta \oplus (\delta(x)u(x)v \otimes \eta)\big)\\
=&(\id{{\c^n}}\otimes J_1)\big( \oplus_{k=1}^n \pi_k \Psi(x) v \otimes \eta \oplus (\oplus_{k=1}^n (\pi_k \otimes \id{{\ltwoi}})\delta(x)u(x)v \otimes \eta)\big)\\
=&(\id{{\c^n}}\otimes J_1)\big( \oplus_{k=1}^n (\pi_k \Psi(x)  v \otimes \eta \oplus (\pi_k \otimes \id{{\ltwoi}})\delta(x)u(x)v \otimes \eta)\big)\\
=&\bigoplus_{k=1}^n J_1 (\pi_k \Psi(x) v \otimes \eta \oplus (\pi_k \otimes \id{{\ltwoi}})\delta(x)u(x)v \otimes \eta)\\
=&\bigoplus_{k=1}^n J_1 p_{k,v , \eta,x}=\bigoplus_{k=1}^n q_{k,v , \eta,x}\\
=&\bigoplus_{k=1}^n \pi_k \Phi(x)v \otimes \eta \oplus (\pi_k \otimes \id{{\ltwoj}})u(x)v \otimes \eta\\
=&\oplus_{k=1}^n (\pi_k \Phi(x) v \otimes \eta \oplus (\pi_k \otimes \id{{\ltwoj}})u(x)v \otimes \eta)\\
=& \oplus_{k=1}^n \pi_k \Phi(x) v \otimes \eta \oplus (\oplus_{k=1}^n (\pi_k \otimes \id{{\ltwoj}})u(x)v \otimes \eta)\\
=& \Phi(x)v \otimes \eta \oplus u(x)v\otimes \eta.
\end{align*}
Now, define 
\beq
v(x)  &\=&  ( \idd - (\idd_{\c^n} \otimes D_1) \delta(x) )^{-1} (\idd_{\c^n} \otimes C_1), \\
\O(x) &\=& (\idd_{\c^n} \otimes A_1) + (\idd_{\c^n} \otimes B_1) \delta(x) v(x),
\qquad \forall x \in \gdel \cap \mnd .
\eeq
Then $\O$ has a free $\delta$-realization, because
\[
\begin{bmatrix}\idd_{\c^n} \ot A_1  & \idd_{\c^n} \ot B_1\\
\idd_{\c^n} \ot C_1&\idd_{\c^n} \ot D_1\end{bmatrix}
\begm
\idd_{\c^n} \otimes \idd_{\k_1} \\
\delta(x) v(x)
\endm
\=
\begm
\O(x) \\
v(x)
\endm, \qquad \forall x \in \gdel \cap \mnd .
\]
It follows from \eqref{eqhy2} that $\O \Psi = \Phi$ on $\gdel$.
\end{proof}

\begin{cor}
\label{corh19}
If $\h$ and $\k_1$ are finite dimensional Hilbert spaces and 
if  $\Phi \in \ball(H^\infty_{\l(\h,\k_1)}(G_\delta))$
then  there exists an isometry
\[
J_1  \=  \begin{bmatrix}A&B\\C&D\end{bmatrix} \ : \h \oplus \ltwo^{(I)} \to \k_1 \oplus \ltwo^{(J)}
\]
so that for $x \in \gdel \cap \mn^d$, 
\be
\label{eqh20}
\Phi(x) \= \idcn\otimes A +  (\idcn \otimes B) \d(x) [ \idcn \otimes \id{{\ltwo^{(J)}}} - (\idcn \otimes D) \d(x) ]^{-1}
\idcn \otimes C.
\ee
Consequently, $\Phi$ has the power series expansion
\be
\label{eqh21}
\Phi(x) \= \idcn \ot A +
\sum_{k=0}^\i (\idcn \ot B) \d(x) [(\idcn \ot D) \d(x) ]^k (\idcn \ot C),
\ee
%\begin{eqnarray}
%\nonumber
%\Phi(x) \=&&\idcn \ot A  \ +\  \sum_{i_1=1}^I  \sum_{j_1 =1}^J (B E_{i_1,j_1} C) \d_{i_1,j_1} (x) \ + \
%\\ 
%\nonumber
%&& \sum_{i_1,i_2=1}^I  \sum_{j_1,j_2 =1}^J (B E_{i_1,j_1}D E_{i_2,j_2} C) \d_{i_1,j_1} (x)
%\d_{i_2,j_2} (x) \ + \ \dots  \ + \\
%&& 
%\sum_{i_1, \dots, i_k=1}^I  \sum_{j_1,\dots,j_k =1}^J (B E_{i_1,j_1}D E_{i_2,j_2} D \dots  E_{i_k,j_k} C) \d_{i_1,j_1} (x) \cdots
%\d_{i_k,j_k} (x)  + \dots 
%\label{eqh21}
%\end{eqnarray}
which is  absolutely convergent on $\gdel$. 
%When the terms are grouped as in \eqref{eqh21}, the resulting series is absolutely convergent on $\gdel$.
%Each group of terms
%\[
%\sum_{i_1, \dots, i_k=1}^I  \sum_{j_1,\dots,j_k =1}^J (B E_{i_1,j_1}D E_{i_2,j_2} D \dots  E_{i_k,j_k} C) \d_{i_1,j_1} (x) \cdots \d_{i_k,j_k} (x) 
%\]
%is bounded in norm by $ \| \d_\E(x) \|^{k-1} $.
\end{cor}
\begin{remark}
If $\h$ and $\k_1$ are both $\c$, then each term
\[
(\idcn \ot B) \d(x) [(\idcn \ot D) \d(x) ]^k (\idcn \ot C) 
\]
is a non-commutative polynomial, whose terms are linear combinations of
products of $k+1$ terms in the entries $\d_{ij}(x)$. 
If one groups the terms by this homogeneity, then the sum of these terms
has norm at most $\| \d(x) \|^{k+1}$.

Corollary~\ref{corh19}, in the case that $\delta(x) = (x^1, \dots, x^d)$,
was proved by Helton, Klep and McCullough \cite[Prop. 7]{hkm12}.
\end{remark}

A special case of Theorem~\ref{thm7.10} is the non-commutative corona theorem.
Take $\h$ and $\k_2$ to be $\c$, and choose $\Phi(x) = \vare$. Then we conclude:

\begin{thm}
\label{thmh3}
Let $\psi_1, \dots, \psi_k$ be in $H^\i(\gdel)$ and satisfy
\[
\sum_{j=1}^k \psi_j(x)^* \psi_j(x) \ \geq \ \vare^2 \, \idcn \qquad
\forall x \in \gdel \cap \mnd.
\]
Then there exist functions $\omega_1, \dots, \omega_k$ in
 $H^\i(\gdel)$ and satisfying $\| (\omega_1, \dots, \omega_k ) \| \leq \frac{1}{\vare}$
in $H^\i_{\l(\c^k,\c)}$ such that
\[
\sum_{j=1}^k \omega_j(x) \psi_j(x) = \idcn  \qquad
\forall x \in \gdel \cap \mnd. \]
\end{thm}
In the case $d=1$ and $\gdel$ is the unit disk, T4heorem~\ref{thmh3} is called
the Toeplitz-corona theorem. It was first proved by Arveson \cite{arv75};
Rosenblum showed how to deduce Carleson's corona theorem from the
Toeplitz corona theorem in \cite{ros80}.

Another consequence of  Theorem~\ref{thm7.10} is the following observation.
Let $ {\mathcal{F}}_\delta$ be the set of $d$-tuples $T$  of commuting operators
satisfying $\| \delta(T) \| \leq 1$. Recall from Definition~\ref{defmay6} that
% Define a norm on holomorphic functions on $\gdel$ by
\be
\label{eqhmay7}
\| f \|_{\delta, {\rm com}} = \sup_{\substack{T \in {\mathcal{F}}_\delta \\ \sigma(T) \subseteq G_\delta}} \norm{f(T)},
\ee
and  $H^\i_{\delta, {\rm com}}$ is the set of analytic functions $f$ on $\gdel$
for which ${\norm{f}}_{\delta,{\rm com}}  < \i$.  (It follows from \cite{amy12b} and
\cite{am13} that the supremum in \eqref{eqhmay7} is the same whether $T$ runs over commuting
operators with Taylor spectrum in $\gdel$ or commuting matrices with a spanning set of joint eigenvectors, and joint eigenvalues that lie in $\gdel$).

Then every free analytic function in 
$H^\i(\gdel)$ has a free $\delta$-realization, and this gives a $\delta$-realization for a
function in $H^\i_{\delta, {\rm com}}$. Conversely, every function in $H^\i_{\delta, {\rm com}}$ has a $\delta$-realization by \cite{at03}, and this extends to a  free $\delta$-realization
for some function $\phi$ in $H^\i(\gdel)$. So we have:

\begin{thm}
\label{thmh4}
Let $$I  \= \{ \phi \in H^\i(\gdel) \, | \, \phi |_{\m^d_1} = 0 \}.
$$
Then 
$H^\i(\gdel)/ I$ is isometrically isomorphic to 
$H^\i_{\delta, {\rm com}}$.
\end{thm}

\section{Oka Representation}
\label{seci1}

\begin{defin}\label{def8.10}
The \emph{free topology} on $\m^d$ is the topology that has as a basis the sets of the form $G_\delta$ where $\delta$ is a matrix of free polynomials in $d$ variables. A \emph{free domain} is a subset of $\m^d$ that is open in the free topology.
\end{defin}
That the definition actually defines a topology follows from the observation that  if $\delta_1$ and $\delta_2$ are matrices of polynomials, then
\[
G_{\delta_1} \cap G_{\delta_2} = G_{\delta_1 \oplus \delta_2}.
\]
An basic property of compact polynomially convex sets in $\c^d$ is that they can be approximated from above by $p$-polyhedrons (cf. \cite{alewer} Lemma 7.4). The following simple proposition asserts that compact sets in the free topology can be approximated from above by polyhedrons as well.
\begin{prop}\label{prop8.10}
Let $E \subseteq \m^d$ be a compact set in the free topology that is closed under (finite)
 direct sums.
 If $U$ is a neighborhood of $E$, and \[
E  \ \subset \ \cup_{\alpha \in A} G_{\delta_\alpha} \subseteq U ,
\]
then there exists  $\delta \in \{ {\delta_\alpha} : \alpha \in A \}$, a single matrix of free polynomials in $d$ variables,  and a positive number $t > 1$,
 such that
\[
E 
\subseteq G_{t\delta} \subseteq
 G_\delta \subseteq U.
\]
\end{prop}
\begin{proof}
Since $E$ is compact and $U$ is open, there are $\d_1, \dots , \d_N$ so that
\[
E \subseteq \cup_{j=1}^N G_{\d_j} \subseteq U .
\]
Claim:
\[
\min_{1 \leq j \leq N} \max_{M \in E} \| \d_j (M) \|  \ < 1.
\]
Indeed, otherwise there would for each $j$ be an $M_j \in E$ such that $\| \d_j(M_j) \| \geq 1$.
Then $\oplus_{j=1}^N M_j$ would be in $E$, but not in any $G_{\d_j}$.

Choose $j$ such that $ \max_{M \in E} \| \d_j (M) \| \= r \ < 1$. Let $\d = \d_j$ and choose $t$
between $1$ and $1/r$.
\end{proof}
\begin{defin}\label{def8.20}
By an $\l(\h,\k)$-valued \emph{free holomorphic function} is meant a graded function 
$\phi:D \to \l(\h,\k)$ such that $D$ is a free domain, $\phi$ is an $\l(\h,\k)$-valued 
graded function on $D$, and for every $M \in D$, there exists a basic free neighborhood $\gdel$ of $M$ in $D$ such that $\phi$ is bounded and nc on $ \gdel$.
\end{defin}
If $\delta$ is a matrix of polynomials, we shall let 
\be
\label{eqiw1}
\kdel \ := \ \{ M \in \m^d : \| \delta(M) \| \leq 1 \}.
\ee
\begin{defin}
\label{defiw2}
Let $E \subset \m^d$. The polynomial hull of $E$ is defined to be
\[
\hat{E} \ := \ \bigcap \{ \kdel \, : \, E \subseteq \kdel \} .
\]
(If $E$ is not contained in any $\kdel$, we declare $\hat{E} $ to be $\m^d$.)
We say a compact set is polynomially convex if it equals its polynomial hull.
We say an open set $D$  is polynomially convex if for any compact set $E \subset D$, 
the polynomial hull of $E$ is a compact subset of $D$.
\end{defin}
Note that $\hat{E}$ is always an nc set, so if  $\hat{E}$ is compact and contained in some
open set $U$, then by Proposition~\ref{prop8.10} it is contained in a single basic free open set in 
$U$.

\begin{example}
Consider the free annulus $A$ 
%defined by \ref{eqay1} (and rescaled for convenience)
\[
A \ := \ \bigcup_{0 \leq \theta \leq 2 \pi}
\{ x \in \m \, : \, \| x -  \frac{3}{4}e^{i\theta} \idd \| < \frac{1}{4} \} .
\]
Suppose $D$ is a polynomially convex free domain containing $A$.
Letting $E$ range over the compact subsets of $\{ z \in \c \ : \frac 12 < |z| < 1 \}$,
and using that $\hat{E} \subset D$, we conclude that $D \cap \m_1 \supseteq \D$,
so $D$ contains all normal matrices with spectrum in $\D$.

For each $r < 1$, we let $E = r \overline{\D}$. By Proposition~\ref{prop8.10},
we conclude that there exists $\d$ such that $\hat{E} \subset \gdel \subset D$.
As $\delta$ is a contractive matrix-valued function on $r \overline{\D}$, it has a realization formula,
and so is contractive on all matrices $M$ with $\| M \| \leq r$.
(Note that in one variable, a polynomial is uniquely defined on $\m$ by its
action on $\m_1 = \c$). 
We conclude therefore that $D$ must contain the open unit matrix ball:
\[
\{ M : \| M \| < 1 \} \ \subseteq \ D .
\]
So polynomial convexity has filled in the holes at all levels.
\end{example}

The following theorem is the free analogue of the Oka-Weil theorem.
\begin{thm}\label{thm8.10}
Let $\h$ and $\k$ be finite dimensional Hilbert spaces.
Let $E \subseteq \m^d$ be a compact set in the free topology,
and assume that $E$ is polynomially convex.
Let $U$ be a free domain containing $E$, and 
let  $\phi$ be a free holomorphic $\l(\h,\k)$-valued function defined on 
$U$.
 Then $\phi$ can be uniformly approximated on $E$ by $\l(\h,\k)$-valued free polynomials.
\end{thm}
\begin{proof}
For each point $M$ in $E$, there is a matrix $\delta_M$ of free polynomials such that
$ M \in G_{\delta_M} \subseteq U$ and $\phi$ is bounded on $G_{\delta_M}$.
By Proposition~\ref{prop8.10}, we can find a single matrix  $\delta$ of free polynomials, and $ t > 1$,
such that  $E 
\subseteq G_{t\delta} \subseteq
 G_\delta$
and such that $\phi$ is bounded on $\gdel$.
 Hence, by Theorem \ref{thm7.10}, $\phi$ has a $\delta$ free realization. Using the resulting Neumann series for $\phi$ (which converges uniformly on $G_{t\delta}$) yields that $\phi$ can be uniformly approximated by polynomials on $E$.
\end{proof}
As an application of Theorem \ref{thm8.10} the following result gives a purely holomorphic characterization of free holomorphic functions. If $\phi$ is a graded function defined on a free domain $D$, let us agree to say that \emph{$\phi$ is locally approximable by polynomials} if for each $M \in D$ and $\epsilon > 0$, there exists a free neighborhood $U$ of $M$ and a free polynomial $p$ such that
\[
\sup_{x \in U\cap D}\norm{\phi(x)-p(x)} < \epsilon.
\]
\begin{thm}\label{thm8.20}
Let $D$ be a free domain and let $\phi$ be a graded function defined on $D$. Then $\phi$ is a free holomorphic function if and only if $\phi$ is locally approximable by polynomials.
\end{thm} 
\begin{proof}
Sufficiency follows because the uniform limit of free polynomials is nc and bounded.
For necessity, let $M$ be in $D$. Then since $D$ is open, there exists  a matrix  $\delta$ of free polynomials, and $ t > 1$,
such that  $M \in
 G_{t\delta} \subseteq
 G_\delta$. Now apply Theorem~\ref{thm8.10} with $E = K_{t\delta}$.
\end{proof}

\section{Free Meromorphic Functions}
\label{secj}

It is a natural question to ask whether rational functions are free holomorphic away from their poles.
A rational function means any function that can be built up from free polynomials by finitely many arithmetic operations. We shall say the polar set of a rational function $\phi$
is the set of $x \in \mnd$ at which, at some stage in the evaluation of the function $\phi(x)$, one
has to divide by a matrix that is not invertible. This obviously depends on the presentation of the function.
We can extend this notion to meromorphic functions on a free open set $D$, defining them
to be  any function that can be built up from free holomorphic functions on $D$ by finitely many arithmetic operations,
and defining their singular set to be the set of points $x \in D$ for which,  at some stage in the evaluation of the function, one
has to divide by a matrix that is not invertible.

\begin{thm}
\label{thmj1}
Let $\phi$ be a meromorphic function on a free domain $D$. Then $\phi$ is free holomorphic off
its singular set.
\end{thm}
\begin{proof}
Since addition, multiplication and scalar multiplication all preserve free holomorphicity, it is sufficient
to prove that if $\phi$ is free holomorphic on $D$ and $\phi(M)$ is invertible
for some point $M$ in $D$, then there is a free open neighborhood of $M$ in $D$
on which $\phi(x)^{-1}$ is bounded.

Since $D$ is open, there exists $\delta$ such that $ M \in \gdel \subseteq D$, and
such that $\phi$ is bounded by $B$ on $\gdel$. Let $T = \phi(M)$, and 
let $p$ be a polynomial in one variable satisfying 
$p(T) = 0, p(0) = 1$.
Let $\delp = \delta \oplus (2 p \circ \phi)$.
If $N \in \gdelp$, then $\| p \circ \phi(N) \| \leq \frac 12$, so
\[
\| [ \idd - p \circ \phi(N) ]^{-1} \| \ \leq \ 2 .
\]
Let $\phi(N) = S$, and let $c, \beta_j \in \C$ satisfy
\[
1 - p(z) \= c z \prod (z - \beta_j) .
\]
Then
\[
[ \idd - p \circ \phi(N) ]^{-1} \= \frac{1}{c}\,  S^{-1} \prod (S - \beta_j)^{-1} .
\]
Therefore
\[
\| S^{-1} \| \ \leq \
2 | c| \prod (B + |\beta_j|),
\]
so $\phi(x)^{-1}$ is bounded on $\gdelp$, as required.
\end{proof}

\section{Index}
\label{secind}

\subsection{Notation}
\label{subsecnot}

{}

$\mn$ The $n$-by-$n$ matrices

 $\md = \cup_{n=1}^\i \mn^d$ Second paragraph, Subsection~\ref{subsecaa}

$\invn$ Invertible $n$-by-$n$ matrices \eqref{eqar1}

$\unin$ Unitary $n$-by-$n$ matrices \eqref{eqar2}

${\rm nc}(D)$ Definition \ref{defa3}

$\gdel = \{ x \in \md: \| \delta(x) \| < 1 \}$ \eqref{eqa3}

$\| f \|_{\delta, {\rm com}}, H^\i_{\delta, {\rm com}}$ Definition \ref{defmay6}

$\l(\h,\k)$ Bounded linear transformations between Hilbert spaces $\h$ and $\k$ Subsection~\ref{subseca2}

$nc_\k(D)$ Line after \eqref{2.100}

$A^\sim$ Envelope of $A$, Definition~\ref{def2.20}

${\rm nc}_{\l(\h,\k)}$ The $\l(\h,\k)$-valued nc functions on $D$, Definition \ref{defax1}

$f^\sim$ \eqref{2.140}

$f^\sss$ \eqref{bx1}

$d(f,g)$ \eqref{3.15}

$\calrm, \calbm, \calgm$ Definition~\ref{def4.10}

$\calb,\calg,\calr$ Paragraph following  Definition~\ref{def4.10}

$\vspace (E,\s), \grade(\calb)$ Paragraph before Proposition \ref{prop4.10}

$X_\calv, X_\calv^r$ \eqref{4.114} and \eqref{4.115}

$X_\calv \stackrel{\rho}{\sim} X_\calb$ \eqref{4.120}

$M_\phi$ \eqref{4.140}, \eqref{5.05}

$H^\i(D), H^\i_{\l(\h,\k)}$ First two paragraphs of Section~\ref{secwo}

$E^{[2]}$  \eqref{5.051}

$\calv_{\l(\h)}, \calv_{\l(\h,\M)} $ Second paragraph, Subsection \ref{ssechv}

$\hvh$ Third paragraph, Subsection \ref{ssechv} 

$\rvh$ \eqref{eqft1}

$\pvh$ \eqref{5.10}

$\cvh, \ctvh$ Definition~\ref{def5.05}

$\mathcal{N}_L$ \eqref{eqxe3}

$\htwole$ line after  \eqref{eqxe3}

$\kdel = \{ x \in \md: \| \delta(x) \| \leq 1 \}$ \eqref{eqiw1}

$\hat{E}$ Definition \ref{defiw2}

\subsection{Definitions}

nc-set, nc-open, nc-closed, nc-bounded, nc-domain: Definition \ref{defa1}

nc-function: Definition \ref{defa3}

basic free open set, free domain, free topology: Definition \ref{defa31}

$\k$-valued nc-function: Definition \ref{defa5}

$\delta$ nc-model: Definition \ref{defa6}

free $\delta$-realization: Definition \ref{defa7}

envelope of $A$: Definition \ref{def2.20}

$\l(\h,\k)$-valued nc functions on $D$, Definition \ref{defax1}

locally bounded, locally uniformly bounded: Definition \ref{def3.10}

partial nc-set: First paragraph, Section~\ref{secpncs}

partial nc-function: Second paragraph, Section~\ref{secpncs}

$\s$-invariant function: Third paragraph, Section~\ref{secpncs}

generic, complete, $E$-reducible: Fourth paragraph, Section~\ref{secpncs}

well-organized pair: Definition \ref{def4.10}

base $\calb$ of well-organized pair:  Paragraph following Definition~\ref{def4.10}

ordered partition: \eqref{4.80}

$\delta$-model on well-organized pairs: Definition \ref{def5.10}

special $\delta$-model: Paragraph before Proposition~\ref{prop5.50}

$\delta$-realization, (partial nc,  $\s$-invariant) on well-organized pairs: Definition \ref{def5.20} 

$\delta$-model (nc, locally bounded, holomorphic): Definition \ref{def6.10}

$\delta$-realization, nc-realization, free realization: Definition \ref{def6.20}

$\l(\h,{\ltwo}^{(J)})$-valued nc-function to order $n$: Definition \ref{def6.30}

polynomial hull, polynomially convex: Definition \ref{defiw2}

\bibliography{../references}

\end{document}